\pdfoutput=1
\documentclass[letterpaper, 11pt]{amsart}
\usepackage{amsmath, amsfonts, amssymb, amsthm, amsrefs, array, stmaryrd, graphicx, hyperref, mathrsfs, eucal, caption}

\usepackage[hhmmss]{datetime}
\usepackage{color}
\newcommand{\kibitz}[2]{\ifnum\Comments=1\textcolor{#1}{#2}\fi}
\theoremstyle{plain}
\newtheorem{thm}{Theorem}[section]
\newtheorem{lemma}{Lemma}[section]
\newtheorem{prop}[lemma]{Proposition}

\theoremstyle{definition}

\theoremstyle{remark}
\newtheorem{remark}[lemma]{Remark}
\numberwithin{equation}{section}

\def\esmath{\ensuremath}

\def\phv{\ensuremath\varphi}

\def\RR{\esmath\mathbb R} %real numbers
\newcommand{\p}{\partial}

\newcommand{\bn}{\begin{enumerate}}
\newcommand{\en}{\end{enumerate}}
\newcommand{\bi}{\begin{itemize}}
\newcommand{\ei}{\end{itemize}}
\newcommand{\bqq}{\begin{eqnarray*}}
\newcommand{\eqq}{\end{eqnarray*}}
\newcommand{\balg}{\begin{align*}}
\newcommand{\ealg}{\end{align*}}

\newcommand{\limt}[2]{\lim\limits_{{#1}\to{#2}}}

\DeclareMathOperator{\Rm}{Rm}

\begin{document}
%%%%%%%%%%%%%%%%%
\title[Noncompact MCF with Type-II curvature blow-up]{Mean curvature flow of noncompact hypersurfaces with Type-II curvature blow-up}

\author{James Isenberg}
\address{Department of Mathematics, University of Oregon, Eugene, OR 97403}
\email{isenberg@uoregon.edu}

\author{Haotian Wu}
\address{School of Mathematics and Statistics, The University of Sydney, NSW 2006, Australia}
\email{haotian.wu@sydney.edu.au}

\thanks{JI is partially supported by NSF grant PHY-1306441.}

%\date{\usdate\today} 

\keywords{Mean curvature flow; noncompact hypersurfaces; Type-II curvature blow-up; precise asymptotics.}

\subjclass[2010]{53C44 (primary), 35K59 (secondary)}

%%%%%%%%%%%%%%%%%
%\linenumbers
%%%%%%%%%%%%%%%%%

\begin{abstract}
We study the phenomenon of Type-II curvature blow-up in mean curvature flows of rotationally symmetric noncompact embedded hypersurfaces. Using analytic techniques based on formal matched asymptotics and the construction of upper and lower barrier solutions enveloping formal solutions with prescribed behavior, we show that for each initial hypersurface considered, a mean curvature flow solution exhibits the following behavior near the ``vanishing'' time $T$: (1) The highest curvature concentrates at the tip of the hypersurface (an umbilic point), and for each choice of the parameter $\gamma>1/2$, there is a solution with the highest curvature blowing up at the rate $(T-t)^{-(\gamma +1/2)}$. (2) In a neighborhood of the tip, the solution converges to a translating soliton which is a higher-dimensional analogue of the ``Grim Reaper'' solution for the curve-shortening flow. (3) Away from the tip, the flow surface approaches a collapsing cylinder at a characteristic rate dependent on the parameter $\gamma$.
\end{abstract}

%% \tableofcontents %% Just for papers exceeding 50 pages.

\maketitle

\section{Introduction}\label{sec:intro}

Let $\phv(t): M^n\to\mathbb{R}^{n+1}$, $t_0<t<t_1$, be a smooth one-parameter family of embeddings (or more generally, immersions) of $n$-dimensional hypersurfaces in $\mathbb{R}^{n+1}$. Mean curvature flow (MCF) evolves the hypersurface $M^n$ in the direction of its mean curvature vector $\vec{H}$ according to the following prescription:
\begin{align}
\label{eq:mcf}
 \p_t \phv(p,t) & = \vec{H}, \quad\quad p\in M^n,\; t_0<t<t_1. 
\end{align}
MCF is a well-studied geometric evolution equation with applications in, for example, material sciences and image processing. It is the negative gradient flow of the volume functional. MCF can be studied, as seen in Brakke's work
\cite{Br78}, in the context of geometric measure theory. It can also be analyzed  from the perspective of partial differential equations (PDEs), as shown initially by Huisken in  \cite{Hui84}.

The mean curvature flow equation \eqref{eq:mcf}, if written out in terms of the components of the evolving mapping $\varphi$, is a (weakly) parabolic system of PDEs. Under this flow, the curvatures of the induced metric satisfy nonlinear reaction-diffusion type equations. Thus, during relatively short time intervals around the time of the initial immersion, MCF tends to smooth out irregularities in the geometry; over a longer time scale, singularities can form. Indeed, finite-time singularities in MCF are known to occur for a wide variety of families of initial data. For example, any closed convex hypersurface remains convex under MCF and shrinks to a ``round point'' in finite time \cite{Hui84}. Within the realm of noncompact immersed hypersurfaces, the mean curvature flow of hypersurfaces that are initially sufficiently close to a cylinder are also known to  become singular in finite time \cite{GS09}.

Singularities which form during the course of mean curvature flow are classified according to the rate of blow-up of the induced second fundamental form of the immersed hypersurface. Specifically, if we let $M_t := \phv(t)(M^n)$ and denote by  $h(p,t)$ the second fundamental form of $M_t$ at $p$, then if $M_t$ evolves by MCF and becomes singular at time $t=T<\infty$, this finite-time singularity is called \emph{Type-I} if
\begin{align*}
 \sup\limits_{p\in M_t}|h(p,t)|(T-t)^{1/2} \leq C
\end{align*}
for some $C<\infty$, and it is called \emph{Type-II} if $\sup\limits_{p\in M_t}|h(p,t)|$ blows up at a faster rate. 

Examples of  the formation of Type-I singularities under MCF are easy to obtain: the mean curvature flows of the standard spheres provide simple examples of Type-I singularities for compact embeddings, while the MCF of the standard cylinders provide such examples for noncompact embeddings. 

Mean curvature flows which develop Type-II singularities are more difficult to specify. One way to do so is to consider a one-parameter family of initial embeddings of the 2-sphere in $\mathbb{R}^3$, with the parameter controlling the extent to which the equator is tightly cinched. For very loose cinching, the flow converges to the shrinking round sphere with its usual global Type I singularity.  For very tight cinching, it has been shown (see \cite{Hui90} and \cite{Ang92} for the case of rotationally symmetric embeddings, and see \cite{GKS, GK} for embeddings which are nearly rotationally symmetric) that the equator shrinks more rapidly than the two ``dumbbell'' hemispheres,  and forms a Type-I ``neckpinch''. To obtain a mean curvature flow which develops a Type-II singularity, one starts the flow at the embedding with the parameter value at the threshold between those embeddings flowing to Type-I neckpinches and those flowing to Type I sphere collapses. The detailed  asymptotics of these Type-II singularities, which develop at the poles of the embedded spheres and have been labeled ``degenerate neckpinches'', have been studied (in the rotationally symmetric case) by Angenent and Vel\'{a}zquez in \cite{AV97}. Rotationally symmetric compact MCF solutions which develop Type-II singularities have been constructed in \cite{AAG}, using the level-set flows of Evans \& Spruck \cite{ES1,ES2,ES3,ES4}, and Chen, Giga \& Goto \cite{CGG}.

In this paper, we study the behavior of Type-II curvature blow-up in mean curvature flows of \emph{noncompact} embedded hypersurfaces.  Part of our motivation for this study comes from the differences which have been observed  between the Type-II singularities which develop in \emph{Ricci flow} on noncompact manifolds and those seen on compact manifolds. A key difference seen in the examples studied thus far concerns the rate of curvature blowup. On compact manifolds $\Sigma$, all the examples that  have been found \cite{AIK11, AIK12} have ``quantized''  blowup rates of $\sup\limits_{x\in \Sigma}|\Rm(x,t)| \sim  (T-t)^{\frac{2}{k}-2}$ for integers $k\ge3$ (here $T$ is the time of the first singularity). By contrast, for noncompact manifolds $\Sigma$, the known examples \cite{Wu14} have a continuous spectrum of blowup rates:  $\sup\limits_{x\in \Sigma}|\Rm(x,t)| \sim (T-t)^{-\lambda-1}$ for all $\lambda \ge 1.$ Noting that the Angenent-Vel\'{a}zquez examples of Type-II singularities in MCF for compact hypersurfaces have quantized rates of blowup \cite{AV97}, we are led to ask what happens for mean curvature flow of noncompact hypersurfaces: Do Type-II singularities exist? If so, what can be said about these singularities?

To motivate how we build mean curvature flows of noncompact hypersurfaces  which exhibit Type-II curvature blow-up, we consider the following setup (see Figure 1):  Suppose we have a graph over an $n$-ball that  asymptotically approaches  the cylinder $S^n\times\mathbb{R}$. If we evolve both the graph and the cylinder via MCF, then both surfaces will shrink:  the cylinder shrinks to a line in finite time, and the evolving graph will move to the right and remain asymptotic to the shrinking cylinder. It follows from the work of S\'{a}ez and Schn\"{u}rer \cite{SS14} that the evolving graph disappears to spatial infinity at the same time as the cylinder collapses. We label this finite time of the graph's disappearance the "vanishing time". Near the vanishing time, the left-most point on the graph must  travel arbitrarily large distances in arbitrarily small amounts of time. Since the evolving graph moves at a speed determined  by its mean curvature, it is plausible that the curvature at the left-most point must blow up ``very fast''. Much of this scenario is confirmed in S\'{a}ez and Schn\"{u}rer's work. However, as noted in Open Problem 1 of \cite{SS14}, the nature of the singular phenomenon at this left-most point, and the behavior of the flow in a neighborhood of the developing singularity, are not resolved in that work. Furthermore, S\'{a}ez and Schn\"{u}rer ask for ``optimal a priori estimates'' in Open Problem 3 of their paper \cite{SS14}. In this paper, we show that mean curvature flows of this sort develop Type-II singular behavior, we show that the allowed rates of decay are not quantized in the sense described above, and we describe the asymptotic behavior of the flows near these singularities.

\begin{figure} \centering
\begingroup%
  \makeatletter%
  \providecommand\color[2][]{%
    \errmessage{(Inkscape) Color is used for the text in Inkscape, but the package 'color.sty' is not loaded}%
    \renewcommand\color[2][]{}%
  }%
  \providecommand\transparent[1]{%
    \errmessage{(Inkscape) Transparency is used (non-zero) for the text in Inkscape, but the package 'transparent.sty' is not loaded}%
    \renewcommand\transparent[1]{}%
  }%
  \providecommand\rotatebox[2]{#2}%
  \ifx\svgwidth\undefined%
    \setlength{\unitlength}{320pt}%
    \ifx\svgscale\undefined%
      \relax%
    \else%
      \setlength{\unitlength}{\unitlength * \real{\svgscale}}%
    \fi%
  \else%
    \setlength{\unitlength}{\svgwidth}%
  \fi%
  \global\let\svgwidth\undefined%
  \global\let\svgscale\undefined%
  \makeatother%
  \begin{picture}(1,0.30695043)%
    \put(0,0){\includegraphics[width=\unitlength,page=1]{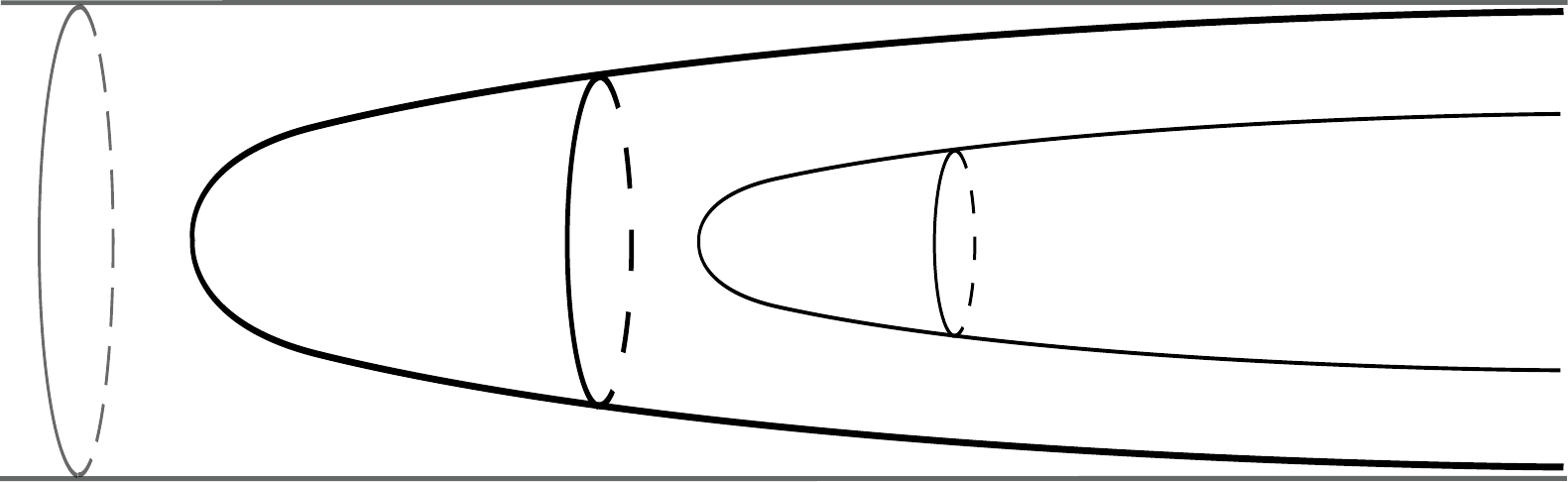}}%
  \end{picture}%
\endgroup%
\caption{}
\end{figure}

To construct these mean curvature flows explicitly, we restrict ourselves to the class of complete hypersurfaces that are rotationally symmetric, (strictly) convex\footnote{Throughout this paper, ``convex'' means ``strictly convex''.}, smooth graphs over a ball and asymptotic to a cylinder. One readily verifies that embeddings with these properties are preserved by MCF (see for example \cite{SS14}). We now introduce the following notation. For any point $(x_0,x_1,\ldots,x_n)\in\mathbb{R}^{n+1}$, we write
\begin{align*}
 x = x_0, \quad r = \sqrt{x_1^2 + \cdots + x_n^2}.
\end{align*}
A noncompact hypersurface $\Gamma$ is said to be rotationally symmetric if
\begin{align*}
 \Gamma = \left\{(x_0,x_1,\ldots,x_n) : r = u(x_0), a\leq x_0 < \infty \right\}.
\end{align*}
We assume that $u$ is strictly concave so that the hypersurface $\Gamma$ is convex and that $u$ is strictly increasing with $u(a)=0$ and with $\lim\limits_{x_0\nearrow \infty}u(x_0) = r_0$, where $r_0$ is the radius of the cylinder. The function $u$ is assumed to be smooth, except at $x=a$. We note that this particular non-smoothness of $u$ is a consequence of the choice of the (cylindrical-type) coordinates; in fact, as seen below, 
if the time-dependent flow function $u(x,t)$ is inverted in a particular way, this irregularity is removed. 
We label the point where $u=0$ the \emph{tip} of the surface.

We now denote by $\Gamma_t$ the solution to MCF which starts at a specified choice of the initial embedding $\Gamma$ (as described above). If we represent $\Gamma_t$ by a graph $r=u(x,t)$, then the function $u$ satisfies the PDE
\begin{align}\label{eq:u(x,t)}
 u_t & = \frac{u_{xx}}{1+u^2_x} - \frac{n-1}{u}.
\end{align}
We know from \cite{SS14} that until the vanishing time $T$, the surface $\Gamma_t$ is asymptotic to the evolving cylinder;  and it follows from Theorem 9.1 and Remark 9.9 (ii) in \cite{SS14} that $\Gamma_t$ races off to $x\nearrow \infty$ at arbitrarily large speed as the cylinder collapses (when $t\nearrow T$).  To determine the details  of this behavior, we analyze solutions of equation \eqref{eq:u(x,t)}.

To help carry out this analysis, especially in a neighborhood of the tip, it is useful to  define the following rescaled quantities 
\begin{align*}
 \tau = -\log(T-t), \quad y = x(T-t)^{\gamma-1/2}, \quad \phi(y,\tau) = u(x,t)(T-t)^{-1/2},
\end{align*}
where $\gamma$ is a parameter to be specified. Substituting these quantities into \eqref{eq:u(x,t)}, we obtain the following PDE for $\phi(y,\tau)$:
\begin{align}\label{eq:phi(y,tau)}
 \left. \p_\tau \right\vert_y \phi & = \frac{e^{-2\gamma\tau} \phi_{yy}}{1+e^{-2\gamma\tau}\phi^2_y} - (1/2-\gamma) y \phi_y
- \frac{(n-1)}{\phi} + \frac{\phi}{2},
\end{align}
where $\left. \p_\tau \right\vert_y $ means taking the partial derivative in $\tau$ while keeping $y$ fixed. We readily note that equation \eqref{eq:phi(y,tau)} admits the constant solution $\phi \equiv \sqrt{2(n-1)}$, which  corresponds to the collapsing cylinder (which is a soliton).

To study more general solutions, it is useful to invert the coordinates and work with 
\begin{align*}
 y( \phi,\tau) & = y\left( \phi(y,\tau), \tau \right);
\end{align*}
this inversion can be done because the hypersurface under consideration is a convex graph over a ball. In terms of $y(\phi,\tau)$, the equation corresponding to mean curvature flow (equivalent to equation \eqref{eq:phi(y,tau)} and hence equivalent to equation \eqref{eq:u(x,t)}) is the following:
\begin{align}\label{eq:y(phi,tau)}
\left. \p_\tau \right\vert_{\phi} y & = \frac{y_{\phi\phi}}{1 + e^{2\gamma\tau} y^2_\phi} + \left(\frac{(n-1)}{\phi} - \frac{\phi}{2} \right) y_\phi + (1/2-\gamma) y.
\end{align}

Our objective in this work is to construct noncompact MCF solutions that exhibit Type-II behavior in finite time and obtain their precise asymptotics. Our main result is the following: (Here we use the notation ``$A \sim B$'' to indicate that there exist positive constants $c$ and $C$ such that $cB \leq A \leq CB$.)

\begin{thm}\label{thmmain}

For any choice of an integer $n\geq 2$ and a pair of real numbers $\gamma > 1/2$, and $\tilde A>0$, there is an open set $\mathscr{G}$ of $n$-dimensional, smooth, complete noncompact, rotationally symmetric, strictly convex hypersurfaces in $\mathbb{R}^{n+1}$ such that the MCF evolution $\Gamma_t$ starting at each hypersurface $\Gamma\in\mathscr{G}$ is trapped in a shrinking cylinder,  escapes at spatial infinity while the cylinder becomes singular at $T<\infty$, and  has the following precise asymptotic properties near the vanishing time $T$ of $\Gamma_t$:
 \bn
 \item The highest curvature occurs at the tip of the hypersurface $\Gamma_t$ , and it blows up at the precise Type-II rate
 \begin{align}
 \sup\limits_{p\in M_t}|h(p,t)| & \sim (T-t)^{-(\gamma+1/2)} \quad \text{as } t\nearrow T.
 \end{align}
  \item Near the tip, the Type-II blow-up of $\Gamma_t$ converges to a translating soliton which is a higher-dimensional analogue of the ``Grim Reaper''
   \begin{align}
 y(e^{-\gamma\tau} z,\tau) & = y(0,\tau) + e^{-2\gamma\tau}\frac{1}{(\gamma-1/2)\tilde A} \tilde P\left( (\gamma-1/2)\tilde A z \right) \quad \text{as } \tau\nearrow \infty
\end{align}
uniformly on compact $z$ intervals, where $z=\phi e^{\gamma\tau}$, and $\tilde P$ is defined below in equation \eqref{eq:tildeP}.
 \item Away from the tip, but near spatial infinity, the Type-I blow-up of $\Gamma_t$  approaches the cylinder at the rate 
 \begin{align}
  2(n-1)-\phi^2 & \sim y^{\frac{1}{1/2-\gamma}} \quad \text{as } y\nearrow \infty.
 \end{align}
 \en
In particular, the solution constructed has the asymptotics predicted by the formal solution described in Section \ref{formal}.

\end{thm}

Roughly speaking, the proof of this theorem (which we carry out in detail below) proceeds as follows. Assuming the hypersurface to be a rotationally symmetric graph (as described above), we find that the MCF equation reduces to a quasilinear  parabolic PDE of a scalar function. Then applying matched asymptotic analysis, we formally construct approximate solutions to the rescaled versions of this PDE. For each such approximate solution, we construct subsolutions and supersolutions which, if carefully patched, form barriers for the rescaled PDE. These barriers carry information of the approximate solution for times very close to the vanishing time. Hence, once we have shown (using a comparison principle) that any solution starting from initial data between the barriers does stay between them for all time, and once we have determined that such initial data sets do exist, we can conclude that there are MCF solutions which exhibit the behavior found in a region near the tip in the approximate solutions. This method of proof based on matched asymptotic analysis has been used in a number of studies of Type-I and Type-II singularities which develop both in Ricci flow  \cite{AIK12, Wu14} and in MCF \cite{AV97, ADS15}.

We note again that while some of the general features of the behavior of hypersurfaces of the form $\Gamma_t$ evolving under MCF have been described in \cite{SS14}, what we do here is to describe the precise asymptotic profile of the solution and quantify the curvature blow-up rate near the vanishing time for the class of complete rotationally symmetric convex graphs. In addition to partially answering Open Problem (1) and shedding light on Open Problem (3) as listed in \cite{SS14}, our result appears to give the first set of examples of \emph{noncompact} solutions to MCF with curvature blowing up at Type-II rates. We note the striking differences of our results from those of Angenent and Vel\'{a}zquez \cite{AV97}, who have constructed a set of MCF solutions on compact hypersurfaces which develop Type-II singularities with discrete ``quantized'' blow up rates of the form $(T-t)^{1/m-1}$ for integer $m\geq 3$. These differences mirror the differences found between Type-II singularities in Ricci flow on noncompact manifolds \cite{Wu14} and those on compact manifolds \cite{AIK11, AIK12}.

Both MCF and Ricci flow are important analytic tools with powerful topological applications. An understanding of the finite-time singularities of these flows has been shown to be very useful in devising protocols for surgery along the flows, which in turn are crucial for using the flows to study the relationship between topological and geometric aspects of manifolds and hypersurfaces.  Perelman's proof of the Poincar\'{e} and Geometrization Conjectures \cite{Per1, Per2} using surgery provides the prime  example of this for Ricci flow, while the analysis of two-convex hypersurfaces by Huisken and Sinestrari \cite{HS09} and the analysis of mean-convex surfaces by Brendle and Huisken \cite{BH15} provide prime examples of this for MCF (see also the work on MCF with surgery by Haslhofer and Kleiner \cite{HK14}).

This paper is organized as follows. Section \ref{formal} describes the construction of the approximate (formal) solutions by the formal matched asymptotics. In Section \ref{super-sub}, we use these approximate solutions to construct the corresponding supersolutions and subsolutions to the rescaled PDE. The supersolutions and subsolutions are ordered and patched to create the barriers to the rescaled PDE in Section \ref{barriers}; a comparison principle for the subsolutions and supersolutions is also proved there. In Section \ref{existence}, we use these results to complete the proof of our main theorem.

%\subsection*{Acknowledgments}

%%%%%%%%%%%%%%%%%
\section{Matched asymptotic analysis and the construction of formal solutions}\label{formal}

As seen in a number of the works which study the detailed asymptotics of the formation of degenerate neckpinches in MCF or Ricci flow \cite{AV97, AIK11, AIK12, Wu14}, a key first step in such a study is to use matched asymptotic analysis to produce formal approximate solutions of the flow. Formal solutions of this sort serve both as templates to which the actual solutions are shown to approach asymptotically, and as guides for the construction of subsolutions and supersolutions.

\subsection{The formal solutions in the form $y(z,\tau)$ or $y(\phi,\tau)$}\label{yztau}

To begin our derivation of a class of formal approximate solutions, we assume that for large values of $\tau$, the terms $\left. \p_\tau \right\vert_{\phi} y$ and $\displaystyle{\frac{y_{\phi\phi}}{1+e^{2\gamma\tau}y^2_\phi}}$ in equation \eqref{eq:y(phi,tau)} are negligible. It follows that the PDE \eqref{eq:y(phi,tau)} can be approximated by the (linear) ODE
\begin{align}
\label{truncd}
 - \left(\frac{(n-1)}{\phi}-\frac{\phi}{2}\right) y_\phi + (\gamma-1/2)y = 0, 
\end{align}
for which the general solution takes the form
\begin{align}\label{eq:truncdsoln}
  \tilde y (\phi) & = C_1 \left[2(n-1) - \phi^2 \right]^{1/2 - \gamma},
\end{align}
where $C_1$ is an arbitrary constant, and $\phi\in[0,\sqrt{2(n-1)})$. If $\gamma>1/2$, then $\tilde y \nearrow \infty$ as $\phi\nearrow\sqrt{2(n-1)}$. This is consistent with the surface being asymptotic to a cylinder at spatial infinity. Since this is a desired feature for the solutions of interest, we henceforth assume that $\gamma>1/2$.

To check the consistency of the assumptions we have made in obtaining the ODE \eqref{truncd}, we substitute the solution $\tilde y$ into the quantity $\displaystyle{\frac{y_{\phi\phi}}{1+e^{2\gamma\tau}y^2_\phi}}$, obtaining 
\begin{align*}
 \frac{\tilde y_{\phi\phi}}{1 + e^{2\gamma\tau} {\tilde y_\phi}^2} & = \frac{2C_1(2\gamma-1)(n-1+\gamma \phi^2)(2n-2-\phi^2)^{\gamma-1/2}}{(2n-2-\phi^2)^{2\gamma+1} + C_1^2  (1-2\gamma)^2  \left(\phi e^{\gamma\tau} \right)^2}.
\end{align*}
This suggests that  $\tilde y$ is a reasonable approximate solution, provided that $\phi e^{\gamma\tau}$ is sufficiently large. We note that for any fixed choice of $\epsilon > 0$, there exists large  $\tau$ such that 
for $\phi \in [\epsilon, 2(n-1)]$, this quantity is in fact large. We label as the \emph{exterior region} the (dynamic) region in which this holds.

To investigate the solution in the complementary \emph{interior region}, where $\phi e^{\gamma\tau} = O(1)$, it is useful to define a new variable 
\begin{align}
z = \phi e^{\gamma \tau}.
\end{align}
We note that the condition $z=O(1)$ is equivalent to the condition $\phi = O\left( e^{-\gamma\tau}\right)$, which corresponds to a region near the tip (at which $\phi = 0$). Using the change-of-variables formula
\begin{align*}
 \p_\tau|_z y = \p_\tau|_{\phi} y - \gamma z y_z,
\end{align*}
together with equation \eqref{eq:y(phi,tau)}, we obtain  the evolution equation for $y(z,\tau)$:
\begin{align}\label{eq:y(z,tau)}
\left. \p_\tau\right\vert_z y & = \frac{y_{zz}}{e^{-2\gamma\tau} + e^{2\gamma\tau}y^2_z} +  \left(e^{2\gamma\tau}\frac{(n-1)}{z} - (\gamma+1/2) z \right) y_z + (1/2-\gamma)y.
\end{align}

To construct a formal solution to equation \eqref{eq:y(z,tau)} in the interior region, we follow the approach used in \cite{AV97} and consider the ansatz
\begin{align}\label{eq:tipansatz-y}
 y = \tilde A + e^{-2\gamma\tau} \tilde F(z,\tau),
\end{align}
where $\tilde A$ is a constant. Substituting this ansatz formula \eqref{eq:tipansatz-y} into equation \eqref{eq:y(z,tau)}, we obtain the PDE
\begin{align}\label{eq:tildeAF}
\frac{\tilde F_{zz}}{1 + \tilde F^2_z} + (n-1) \frac{\tilde F_z}{z} & = (\gamma-1/2) \tilde A + e^{-2\gamma\tau}\left[ (\gamma+1/2)(z \tilde F_z - \tilde F) + \left.\p_\tau\right\vert_z \tilde F \right]. 
\end{align}
Continuing the formal argument, we assume that for $\tau$ very large, the term in \eqref{eq:tildeAF} with the coefficient $e^{-2\gamma\tau}$ is negligible. Equation \eqref{eq:tildeAF} then reduces  to the ODE
\begin{align}\label{eq:tildeF}
 \frac{\tilde F_{zz}}{1+\tilde F^2_z} + (n-1)\frac{\tilde F_z}{z} & =  (\gamma-1/2) \tilde A.
\end{align}

To solve \eqref{eq:tildeAF}  for $\tilde F$, we define $\tilde P(z)$ to be the unique solution to the initial value problem 
\begin{align}\label{eq:tildeP}
 \frac{\tilde P''}{1+(\tilde P')^2} + (n-1)\frac{\tilde P'}{z} = 1,\quad  \tilde P(0) = \tilde P'(0) = 0.
\end{align}
We then readily verify that if  $\tilde F$ is given by
\begin{align}\label{eq:tildeF(z,tau)}
 \tilde F(z,\tau) = \frac{1}{(\gamma-1/2) \tilde A} \tilde P\left( (\gamma-1/2) \tilde A z \right) + C(\tau),
\end{align}
where $C(\tau)$ is an arbitrary function of time, then $\tilde F$ satisfies \eqref{eq:tildeF}. We note the following: 

\begin{remark}
For dimension $n=1$, equation \eqref{eq:tildeP} reduces to
\begin{align*}
\frac{\tilde P''}{1+(\tilde P')^2} = 1,\quad  \tilde P(0) = \tilde P'(0) = 0,
\end{align*}
whose solution is $\tilde P(z) = -\log \cos z$. The graph of $x= t-\log \cos z$, for $z\in(-\pi/2,\pi/2)$, has been labeled the ``Grim Reaper''. It translates with constant velocity along the $x$-axis and is a solution to the curve-shortening flow (i.e., 1-dimensional MCF). For $n\geq 2$, solving \eqref{eq:tildeP} for the function $\tilde P$ corresponding to $n$ dimensions and then rotating the graph of $x=c^{-1}\tilde{P}(cz)+ct$ around the $x$-axis defines a higher dimensional analog of the ``Grim Reaper''. This MCF solution is a translating soliton. 
\end{remark}

The initial value problem \eqref{eq:tildeP} has been solved in \cite[pp.24--25]{AV97} for general diimensions. It has a unique convex solution with the following asymptotics:
\begin{align}\label{eq:asymptotics-tildeP}
 \tilde P(z) &= \left\{
  \begin{array}{cc}
   z^2/2n + o\left(z^2\right), & z\searrow 0;\\ \\
   z^2/(2n-2) - \log z + O\left(z^{-2}\right),& z\nearrow\infty.
     \end{array}
  \right.
\end{align}
It then follows that the asymptotics for $y(z,\tau)$ take the form
\begin{align}\label{eq:asymptotics-y}
y(z) &= \left\{
  \begin{array}{cc}
   \tilde A + e^{-2\gamma\tau} C(\tau) + e^{-2\gamma\tau} (\gamma-1/2)\tilde A \frac{z^2}{2n} + o\left(e^{-2\gamma\tau}z^2\right), & z\searrow 0;\\ \\
   \vspace{10pt}
   \tilde A + e^{-2\gamma\tau} C(\tau) + e^{-2\gamma\tau} (\gamma-1/2)\tilde A \frac{z^2}{2(n-1)} + O\left(e^{-2\gamma\tau}\log z\right),& z\nearrow\infty.
     \end{array}
  \right.
\end{align}
Recalling the scaling formulas $x = y (T-t)^{1/2-\gamma}$ and $z = u (T-t)^{-\gamma - 1/2}$, as well as the interior region ansatz formula  \eqref{eq:tipansatz-y} and the expression \eqref{eq:asymptotics-tildeP} for the asymptotics of $\tilde P$,  we obtain the following asymptotic expression for $x$ in a neighborhood of the tip (at $z=0$):
\begin{align*}
 x & = y (T-t)^{1/2-\gamma} = y e^{(\gamma-1/2)\tau} \\
   & =  e^{(-\gamma-1/2)\tau}C(\tau) + \frac{\tilde{A}}{(T-t)^{\gamma-1/2}} + \frac{(\gamma-1/2)\tilde A}{(T-t)^{\gamma+1/2}} \frac{u^2}{2n} + o\left((T-t)^{(-\gamma-1/2)}u^2 \right).
\end{align*}

The highest curvature occurs at the tip (as we verify below in Lemma \ref{tip}), so the mean curvature and hence the normal (horizontal)  velocity attain their maximal values there. Substituting in the asymptotic expression for $x$ obtained above, we have 
\begin{align}\label{eq:meancurv}
 \left. H \right\vert_{\text{tip}} & = n \left. \frac{d^2 x}{du^2}\right   \vert_{u=0} = \frac{(\gamma-1/2)\tilde A}{(T-t)^{\gamma+\frac{1}{2}}}.
\end{align}
If $\gamma>1/2$ (which we presume throughout this work), then it follows from \eqref{eq:meancurv} that the curvature blows up at Type-II rate. Moreover, over the time period $[t_0,T)$, the tip of the surface moves  along the $x$-axis to the right from its initial position $x_0$ by the amount
\begin{align*}
 \int_{t_0}^{T} H \; ds & = \lim\limits_{t\nearrow T} \int_{t_0}^{t} H \; ds\\
    & = \lim\limits_{t\nearrow T} \frac{(\gamma-1/2)\tilde A}{(T-t)^{\gamma-1/2}} - x_0 \\
    & = \infty.
\end{align*}
Hence, we see that in terms of the $x$-coordinate, the surface evolving by MCF disappears off to spatial infinity as $t\nearrow T$. However in terms of  the $y$-coordinate, so long as one chooses $C(\tau)=O(\tau)$ (cf. Section 3.2), one finds that the tip remains a finite distance from the origin for all time $\tau$ since
\begin{align*}
y_0(\tau) & = \tilde A + e^{-2\gamma\tau}C(\tau),\\
  & \approx \tilde A.
\end{align*}

The formal solutions constructed separately in the interior and the exterior regions each involves a free parameter. Matching the formal solutions on the overlap of the two regions, we can establish an algebraic relationship between them. 
Using the large $z$ asymptotic expansion formula \eqref{eq:asymptotics-y} for the solution $y(z)$ in the interior region, one has (setting $z$ equal to a constant $R$, and presuming that $\tau$ is very large) 
\begin{align}\label{eq:match-formal-int}
 y & = \tilde A + e^{-2\gamma\tau} C(\tau)  e^{-2\gamma\tau} \frac{(\gamma-1/2)\tilde A}{2(n-1)} R^2 + O\left(e^{-2\gamma\tau}\log R \right), \notag\\
   & \approx \tilde A.
\end{align}
In the exterior region, again setting $z=R$ (and therefore $\phi = Re^{-\gamma \tau}$) and again presuming very large $\tau$, we have
\begin{align}\label{eq:match-formal-ext}
   y & = C_1\left[2(n-1)-\phi^2\right]^{1/2-\gamma} \notag\\
     & \approx C_1\left[2(n-1)\right]^{1/2-\gamma}.
\end{align}
Matching \eqref{eq:match-formal-int} with \eqref{eq:match-formal-ext}, we obtain 
\begin{align*}
 \tilde A \approx C_1\left[2(n-1)\right]^{1/2-\gamma}.
\end{align*}

We now collect these results and write out expressions for our formal solutions, both in the interior and the exterior regions. We fix the constant $\gamma >1/2$ (one of the free parameters in the formal solution). In the interior region, which is characterized by $z = \phi e^{\gamma\tau} = u e^{(\gamma+1/2)\tau} = O(1)$, we blow up the  MCF solution $u(t,x)$  at the prescribed Type-II rate $(T-t)^{-\gamma-1/2}$. We also rescale the coordinates in accord with how fast the surface moves under mean curvature flow by setting $y = x (T-t)^{\gamma-1/2}$. Then in this interior region (see Lemma \ref{interior-supersub} for the precise definition of the interior region), the formal solution is given by 
\begin{align*}
y_{form, int} & = \tilde A +  e^{-2\gamma\tau} C(\tau) + e^{-2\gamma\tau} \tilde F(z),
\end{align*}
where $\tilde F$ and the as-yet-unspecified function $ C(\tau)$ are related to $\tilde P$ as specified in \eqref{eq:tildeF(z,tau)}, and where $\tilde P$ is the solution to the initial value problem \eqref{eq:tildeP}.

In the exterior region, where $R e^{-\gamma\tau} \leq \phi < \sqrt{2(n-1)}$ for some $R>0$ (see Lemma \ref{exterior-supersub} for the precise definition of the exterior region), the formal solution takes the form
\begin{align*}
 y_{form, ext} & = \frac{\tilde A}{[2(n-1)]^{1/2 -\gamma}} \left[2(n-1) - \phi^2 \right]^{1/2 - \gamma}.
\end{align*}
We note that as $\phi\nearrow\sqrt{2(n-1)}$, one has $w\nearrow\infty$, which indicates that the exterior formal solutions are  asymptotic to and lie strictly within the cylinder of radius $\sqrt{2(n-1)}$.

%%%%%%%%%%%%%%%%%%%%%%%%%%%%%%%%%

\subsection{The formal solutions revisited in the form  $\lambda(z,\tau)$ or $\lambda(\phi, \tau)$}\label{lambdaztau}

To prove the main result (Theorem \ref{thmmain}) of this paper, it is useful to work with the quantity $\lambda := -1/y$, since in terms of $\lambda$, the asymptotically cylindrical end of the embedded hypersurface corresponding to large values of $y$ is effectively compactified. The MCF evolution equation for $\lambda$ is readily obtained by substituting the definition of $\lambda$ into \eqref{eq:y(phi,tau)}:
\begin{align}\label{eq:lambda(phi,tau)}
 \left. \p_\tau \right\vert_\phi \lambda & = \frac{\lambda_{\phi\phi}-2\lambda^2_\phi/\lambda}{1+e^{2\gamma\tau}\lambda^2_\phi/\lambda^4} + \left(\frac{n-1}{\phi} - \frac{\phi}{2}\right)\lambda_\phi + (\gamma-1/2)\lambda.
\end{align}
The class of MCF solutions we consider here correspond to (convex) solutions of equation \eqref{eq:lambda(phi,tau)} subject to the following effective boundary conditions: the rotational symmetry of the graph implies that $\lambda_\phi (0,\tau) = 0$, and the asymptotically cylindrical condition implies that $\lambda(\sqrt{2(n-1)}, \tau) = 0$. Working with the even-extension of $y(\phi, \tau)$, we find that the MCF solutions of interest also correspond to (convex) even solutions to equation \eqref{eq:lambda(phi,tau)} subject to the boundary conditions $\lambda(-\sqrt{2(n-1)}, \tau) = \lambda(\sqrt{2(n-1)}, \tau) = 0$.

As in the analysis done above (in Section \ref{yztau}) in terms of $y$, it is useful here in working with $\lambda$ to use the dilated spatial variable $z=\phi e^{\gamma\tau}$. The evolution equation for $\lambda(z,\tau)$ then takes the form
\begin{align}\label{eq:lambda(z,tau)}
 \left. \p_\tau \right\vert_z \lambda & = \frac{e^{2\gamma\tau}(\lambda_{zz}-2\lambda^2_z/\lambda)}{1+e^{4\gamma\tau}\lambda^2_z/\lambda^4} + e^{2\gamma\tau}(n-1)\frac{\lambda_z}{z} - (\gamma+1/2)z\lambda_z + (\gamma-1/2)\lambda.
\end{align}

We now construct the formal solutions in terms of $\lambda(z,\tau)$ or $\lambda(\phi,\tau)$, using arguments very similar to those used above in constructing the formal solutions in  terms of $y(z,\tau)$ or $y(\phi,\tau)$. 

In the interior region, where we expect $z=O(1)$, we use the ansatz
\begin{align*}
\lambda=-A+e^{-2\gamma\tau}F(z),
\end{align*}
where $A$ is a positive constant. Substituting this ansatz into equation \eqref{eq:lambda(z,tau)}, we find that 
$F$ must satisfy
\begin{align}
\label{Flambda}
 e^{-2\gamma\tau}\left( -2\gamma F \right) & = \frac{F_{zz} - 2e^{-2\gamma\tau} F^2_z/(-A+e^{-2\gamma\tau}F)}{1+F^2_z/(-A+e^{-2\gamma\tau} F)^4} + (n-1)\frac{F_z}{z}\\
 \nonumber &\quad  - e^{-2\gamma\tau}(\gamma+1/2) z F_z - (\gamma-1/2)A + e^{-2\gamma\tau}(\gamma-1/2)F.
\end{align}
Assuming, for the sake of the formal argument, that the terms in equation \eqref{Flambda} with the coefficient $e^{-2\gamma\tau}$ in equation \eqref{Flambda} can be neglected if $\tau$ is large, we find that \eqref{Flambda} reduces to the following ODE for $F$:
\begin{align}\label{eq:ODE-F}
 \frac{F_{zz}}{1+F^2_z/A^4} + (n-1) F_z/z = (\gamma-1/2)A.
\end{align}
To find solutions to \eqref{Flambda}, we rescale $F$ according to
\begin{equation}
\label{FP}
F(z)= \frac{A^3}{\gamma-1/2} P(\bar z),
\end{equation}
where $\bar z := z(\gamma-1/2)/A$, and determine that $P$ satisfies the ODE
\begin{align*}
 \frac{P_{\bar{z}\bar{z}}}{1+(P_{\bar{z}})^2} + (n-1)\frac{P_{\bar{z}}}{\bar z} = 1.
\end{align*}
Subject to the initial conditions $P(0)=P'(0)=0$, we can solve for $P$ uniquely (cf. equation \eqref{eq:tildeP}). Moreover, the asymptotic expansions of $P(\bar z)$ are known (see \cite[pp.24-25]{AV97}):
\begin{align*}
 P(\bar z) &= \left\{
  \begin{array}{cc}
   \frac{1}{2n} \bar z^2 + o\left(\bar z^2\right), & \bar z\searrow 0;\\ \\
   \frac{1}{2(n-1)} \bar z^2 - \log \bar z + O\left(\bar z^{-2} \right),& \bar z\nearrow\infty.
     \end{array}
  \right.
\end{align*}
Consequently, the asymptotic expansions of $F(z)$ are as follows:
\begin{align}\label{eq:asymp-F}
 F(z) &= \left\{
  \begin{array}{cc}
   \frac{(\gamma-1/2)A}{2n} z^2 + o\left(z^2\right), & z\searrow 0;\\ \\
   \frac{(\gamma-1/2)A}{2(n-1)} z^2 - \frac{A^3}{(\gamma-1/2)}\log( (\gamma-1/2)A^{-1} z) + O\left( z^{-2}\right),& z\nearrow\infty.
     \end{array}
  \right.
\end{align}

In the exterior region, examining  the evolution of $\lambda(\phi,\tau)$ as governed by the PDE \eqref{eq:lambda(phi,tau)}, we assume (as part of the formal argument) that the term $\displaystyle{\frac{\lambda_{\phi\phi}-2\lambda^2_\phi/\lambda}{1+e^{2\gamma\tau}\lambda^2_\phi/\lambda^4}}$ is negligible for $\tau$ large. Further, we note that any stationary solution of the equation
\begin{align}\label{eq:ODE-lambdabar}
 \left(\frac{n-1}{\phi} - \frac{\phi}{2} \right)\bar\lambda_\phi + (\gamma-1/2)\bar\lambda = 0
\end{align}
is an approximate solution to $\lambda(\phi,\tau)$. We can solve for $\bar\lambda(\phi)$ explicitly, obtaining 
\begin{align}\label{eq:lambdabar}
 \bar\lambda(\phi) & = C [2n-2-\phi^2]^{\gamma-1/2}
\end{align}
for an arbitrary constant $C$.

%%%%%%%%%%%%%%%%%%%%%%%%%%%%%%%%%

%%%%%%%%%%%%%%%%%%%%%%%%%%%%%%%%%

\section{Supersolutions and subsolutions}\label{super-sub}

For a differential equation of the form $\mathcal{D}[\psi]=0$, a function $\psi_+$ is a \emph{supersolution} if $\mathcal{D}[\psi^+]\ge 0$, while $\psi^-$ is a \emph{subsolution} if $\mathcal{D}[\psi^-]\le 0$. If there exist a supersolution $\psi^+$ and a subsolution $\psi_-$ for the differential operator $\mathcal{D}$, and if they satisfy the inequality $\psi^+ \ge \psi^-$, then they are called \emph{upper and lower barriers}, respectively. If $\mathcal{D}[\psi]=0$ admits solutions, then the existence of barriers $\psi^+\geq \psi^-$ implies that there exists a solution $\psi$ with $\psi^+ \ge \psi \ge \psi^-$. 

In this section, we construct subsolutions and supersolutions for the MCF of our models. We do this first separately in the interior and the exterior regions, and then (in the next section) combine these results to obtain subsolutions and supersolutions for the flow of a complete hypersurface.

\subsection{Interior region}

In the interior region, we work with  $\lambda(z,\tau)$, and with the corresponding  MCF equation \eqref{eq:lambda(z,tau)}. Hence, we work with the quasilinear parabolic operator
\begin{align}
\label{Tz}
 \mathcal{T}_z[\lambda] : = & \left. \p_\tau \right\vert_z \lambda - \frac{e^{2\gamma\tau}(\lambda_{zz}-2\lambda^2_z/\lambda)}{1+e^{4\gamma\tau}\lambda^2_z/\lambda^4} - e^{2\gamma\tau}(n-1)\frac{\lambda_z}{z} + (\gamma+1/2)z\lambda_z + (1/2 - \gamma)\lambda,
\end{align}
for which we seek subsolutions and supersolutions. We begin the construction of these as follows.

\begin{lemma}\label{interior-supersub}
For an integer $n\geq 2$, a real number $\gamma>1/2$, and a pair of real numbers $A^+, A^->0$, we define the function $F^{+}$ (or $F^-$) to be the solution to equation \eqref{eq:ODE-F} with the constant $A$ in that equation set to be $A^{+}$ (or $A^-$).

For any fixed constants $R_1>0$, $B^+, B^-$, and $E^+, E^-$, there exist functions $Q^+, Q^-:\RR\to\RR$, constants $D^+, D^-$, and a sufficiently large $\tau_1<\infty$ such that the functions
 \begin{align}\label{eq:int-supersub}
\lambda^{\pm}_{int} = \lambda^{\pm}_{int} (z,\tau) := -A^{\pm} + e^{-2\gamma\tau}F^{\pm}(z) + e^{-2\gamma\tau} \left(B^{\pm}\tau + E^{\pm}\right)  + \tau e^{-4\gamma\tau}D^{\pm}Q^{\pm}(z)
\end{align}
are supersolution ($+$) and subsolution ($-$), respectively, of $\mathcal{T}_z[\lambda]=0$ on the interval $0 \leq |z| \leq R_1$ for all $\tau\geq\tau_1$.

The function $Q^{+}$ (or $Q^{-}$) depends on $A^{+}$ and $F^{+}(z)$ (or on $A^{-}$ and $F^-(z)$); the constant $D^{+}$ (or $D^{-}$) depends on $n$, $\gamma$, $A^{+}, B^{+}$ (or $A^{-}, B^{-}$), and $R_1$.
\end{lemma}

\begin{proof}
We detail the proof for $0\leq z \leq R_1$; the proof for $-R_1 \leq z \leq 0$ is essentially identical.

Let us define $Q^{+}$ to be the unique solution of the ODE initial value problem
\begin{align}\label{eq:ODE-Q}
  & -(A^{+})^{-2} \left[ \frac{Q'}{1+F^2_z/(A^{+})^4}\right]' - (n-1)\frac{Q'}{z} = 1,\\
  & Q(0)=Q'(0)=0. \notag
\end{align}
The function $Q^{-}$ is defined by replacing $A^{+}$ with $A^{-}$ in the above equation.

Applying the operator $\mathcal{T}_z$ defined above (see \eqref{Tz}) to the function $\lambda ^+_{int}$ from \eqref{eq:int-supersub}, we calculate
\begin{align*}
\mathcal{T}_z[\lambda^{+}_{int}] & = I + II + III + IV + V,
\end{align*}
where (to simplify the expressions, we replace ``$\lambda^+_{int}$'' by ``$\lambda$'' and omit the superscript ``$+$'' in the definition of $\lambda^+_{int}$)
\begin{align*}
  I & = \left. \p_\tau \right\vert_\phi \lambda\\
    & = -2\gamma B \tau e^{-2\gamma\tau} + B e^{-2\gamma\tau} - 2\gamma E e^{-2\gamma\tau},\\
III & = - e^{2\gamma\tau}(n-1) \frac{\lambda_z}{z}\\
    & = -(n-1)\frac{F_z}{z} - (n-1) \frac{Q'}{z} D\tau e^{-2\gamma\tau},\\
 IV & = (\gamma+1/2)z\lambda_z \\
    & = (\gamma+1/2)(zF_z)e^{-2\gamma\tau} + \tau e^{-4\gamma\tau} D Q',\\
  V & = (1/2-\gamma)\lambda \\
    & = (\gamma-1/2)A + (1/2-\gamma) B\tau e^{-2\gamma\tau} \\
    &\quad + (1/2-\gamma)\left[E e^{-2\gamma\tau} + e^{-2\gamma\tau}F(z) + D \tau e^{-4\gamma\tau}Q(z)\right],
\end{align*}
and
\begin{align*}
II & = - \frac{e^{2\gamma\tau}(\lambda_{zz}-2\lambda^2_z/\lambda)}{1+e^{4\gamma\tau}\lambda^2_z/\lambda^4}\\
   & = -\lambda^2\left[\arctan\left( e^{2\gamma\tau} \frac{\lambda_z}{\lambda^2} \right) \right]_z.
\end{align*}
To obtain an estimate for $II$, we calculate
\begin{align*}
 e^{2\gamma\tau} \frac{\lambda_z}{\lambda^2} & = \frac{F_z + \tau e^{-2\gamma\tau} D Q'}{A^2}\left[1 + O\left(\pi_1(l.o.t.) \right)\right],
\end{align*}
where $\pi_1(l.o.t.)$ is a polynomial (without constant term) in $\tau e^{-2\gamma\tau} B$, $e^{-2\gamma\tau} E$, $e^{-2\gamma\tau} F$, and $\tau e^{-4\gamma\tau} Q$; hence, we obtain 
\begin{align*}
II & = -\lambda^2\left[\arctan\left( e^{2\gamma\tau} \frac{\lambda_z}{\lambda^2} \right) \right]_z\\
   & = -\frac{F_{zz}}{1+F^2_z/A^4} - A^{-2} \left[\frac{ Q'}{1+F^2_z/A^4}\right]' D \tau e^{-2\gamma\tau} \\
   & \quad - \frac{2A^{-1}B F_{zz}}{1+F^2_z/A^4} \tau e^{-2\gamma\tau} + O\left( \pi_2(l.o.t.)\right),
\end{align*}
where $\pi_2(l.o.t.)$ is a polynomial (without constant term) in $e^{-2\gamma\tau} E$, $e^{-2\gamma\tau} F$, $\tau e^{-4\gamma\tau} Q$, and $\tau^2 e^{-4\gamma\tau} Q'Q''$.

It follows from the asymptotic expansion \eqref{eq:asymp-F} of $F$ that there exists constant a $C=C(R_1)>0$ such that $\left\vert F(z) \right\vert \leq C$ and $\displaystyle{ \left\vert \frac{2A^{-1}B F_{zz}}{1+F^2_z/A^4} \right\vert} \leq BC $ for $z\in[0, R_1]$. Then using equations \eqref{eq:ODE-F} and \eqref{eq:ODE-Q} which are satisfied by $F$ and $Q$, respectively, we have for $0\leq z \leq R_1$ and $\tau\geq \tau_1$ with $\tau_1$ sufficiently large,
\begin{align*}
e^{2\gamma\tau}\mathcal{T}_z[\lambda] & = \left\{D - \left(3\gamma-1/2\right)B - \frac{2A^{-1}B F_{zz}}{1+F^2_z/A^4} \right\}\tau + \left(E + F(z) + O(\tau e^{-2\gamma\tau})\right)\\
& \geq  \left\{D - \left(3\gamma-1/2 + C\right)B  \right\}\tau + O(1)\\
& > 0;
\end{align*}
this last inequality holds so long as we choose 
 $D$ so that $D > (3\gamma-1/2 + C)B$;  that is,  $D^+ > (3\gamma-1/2 + C)B^+$.

By a similar argument, there exists $D^{-} < (3\gamma-1/2 - C)B^{-}$ such that $\mathcal{T}_z[\lambda^-]<0$ for $0 \leq z \leq R_1$ and $\tau\geq\tau_1$.

The lemma is therefore proven.
\end{proof}

\subsection{Exterior region}

In the exterior region, we work with the quantity $\lambda(\phi,\tau)$, and with the corresponding MCF equation \eqref{eq:lambda(phi,tau)}. Hence, defining the quasilinear parabolic operator
\begin{align}\label{eq:Fphi}
 \mathcal{F}_\phi[\lambda] &: =  \left. \p_\tau \right\vert_\phi \lambda - \frac{(\lambda_{\phi\phi}-2\lambda^2_\phi/\lambda)}{1+e^{2\gamma\tau}\lambda^2_\phi/\lambda^4} - \left(\frac{n-1}{\phi}-\frac{\phi}{2}\right)\lambda_\phi + (1/2 - \gamma)\lambda,
\end{align}
we seek sub and supersolutions for this operator. The existence of these is proven in the following:

\begin{lemma}\label{exterior-supersub}
For an integer $n\geq 2$ and a real number $\gamma>1/2$, we define\footnote{This definition is consistent with \eqref{eq:lambdabar}; therefore $\bar\lambda$ satisfies equation \eqref{eq:ODE-lambdabar}.}
\begin{align}\label{eq:barlambda}
\bar\lambda = \bar\lambda(\phi) := (2n-2-\phi^2)^{\gamma-1/2}.
\end{align}

For any fixed set of constants $R_2, c^+, c^->0$, there exist a function \\ $\psi:(-\sqrt{2(n-1)},\sqrt{2(n-1)}) \to \RR$, a pair of constants $b^+$ and  $b^-$, and sufficiently large $\tau_2<\infty$ such that the functions
\begin{align}\label{eq:ext-supersub}
 \lambda^{\pm}_{ext} & = \lambda^{\pm}(\phi,\tau) := -c^{\pm} \bar\lambda(\phi) + b^{\pm} e^{-2\gamma\tau}\psi(\phi)
\end{align}
are supersolution ($+$) and subsolution ($-$), respectively, of $\mathcal{F}_\phi[\tilde\lambda]=0$ on the interval\\
 $R_2 e^{-\gamma\tau} \leq \vert\phi\vert < \sqrt{2(n-1)}$ for all $\tau\geq\tau_2$. The constant $b^+$ (or $b^-$) depends on $n, \gamma$, $R_2$, and $c^+$ (or $c^{-}$).
\end{lemma}

\begin{proof}
We carry out  the proof for $\phi\in[R_2 e^{-\gamma\tau}, \sqrt{2(n-1)})$; the proof for \\ $\phi\in(-\sqrt{2(n-1)}, -R_2 e^{-\gamma\tau}]$ is essentially identical.

Applying the operator $\mathcal{F}_\phi$ defined in \eqref{eq:Fphi} to the function $\lambda^{+}_{ext}$ from \eqref{eq:ext-supersub}, we obtain 
\begin{align*}
 e^{2\gamma\tau}\mathcal{F}_\phi[\lambda^+_{ext}] & = I + II + III + IV,
\end{align*}
where (as in Lemma \ref{interior-supersub}, to simplify the notation we omit the superscript ``$+$'' and the subscript ``ext'')
\begin{align*}
 I & = e^{2\gamma\tau} \left. \p_\tau \right\vert_\phi \lambda  = -2\gamma b \psi,\\
 II & = - \frac{(\lambda_{\phi\phi}-2\lambda^2_\phi/\lambda)}{e^{-2\gamma\tau}+\lambda^2_\phi/\lambda^4},\\
 III & = - e^{2\gamma\tau} \left(\frac{n-1}{\phi}-\frac{\phi}{2}\right)\lambda_\phi \\
     & = - e^{2\gamma\tau}\left(\frac{n-1}{\phi}-\frac{\phi}{2}\right)c \bar\lambda' - \left(\frac{n-1}{\phi}-\frac{\phi}{2}\right) b\psi',\\
 IV & = e^{2\gamma\tau}(1/2 - \gamma)\lambda \\
    & = e^{2\gamma\tau}(1/2-\gamma)c\bar\lambda + (1/2-\gamma)  b\psi,
\end{align*}
where $\bar\lambda$ solves equation \eqref{eq:ODE-lambdabar}. 

Using \eqref{eq:ODE-lambdabar} and combining, we have 
\begin{align*}
 e^{2\gamma\tau}\mathcal{F}_\phi[\lambda^+_{ext}] & = II + b\left\{ (1/2 - 3\gamma)\psi - \left(\frac{n-1}{\phi}-\frac{\phi}{2}\right) \psi' \right\}.
\end{align*}
If we set $\Lambda = \Lambda(\phi) := - \left[(-\bar\lambda'') - 2(-\bar\lambda')^2/(-\bar\lambda)\right] \left[(-\bar\lambda)^4/(-\bar\lambda')^2\right]$, we calculate (using \eqref{eq:barlambda})
\begin{align*}
 \Lambda(\phi) & = - \frac{2}{\gamma-1/2} \frac{(n-1+\gamma\phi^2)}{\phi^2}(2n-2-\phi^2)^{3(\gamma-1/2)},
\end{align*}
and note that $\Lambda < 0$ for $0<\phi<\sqrt{2(n-1)}$. 

We now define the function $\psi$ to be any solution of the ODE
\begin{align}\label{eq:ODE-psi}
 (1/2 - 3\gamma) \psi - \left( \frac{n-1}{\phi} - \frac{\phi}{2} \right)\psi' & = \Lambda(\phi).
\end{align}
This ODE can be solved explicitly; the general solution $\psi$ is given by
\begin{align} \label{eq:soln-psi}
 \psi & =  \frac{2\gamma+1}{2\gamma-1}(2n-2-\phi^2)^{3\gamma-\frac{3}{2}} \\
 & \quad  + \frac{(2n-2-\phi^2)^{3\gamma-\frac{1}{2}}}{2(n-1)(2\gamma-1)}\left[ C_1 + \log(\phi^2) - \log(2n-2-\phi^2)\right] \notag
\end{align}
for an arbitrary constant $C_1$. It follows then that
\begin{align*}
 e^{2\gamma\tau}\mathcal{F}_\phi[\lambda^+_{ext}] & = II + b \Lambda(\phi).
\end{align*}

We now estimate the term $II$. Since (as follows from  \eqref{eq:ext-supersub})
\begin{align*}
 \lambda & = (-c\bar\lambda) \left[ 1 - e^{-2\gamma\tau} b\psi/(c\lambda)\right],\\
 \lambda_\phi & = (-c\bar\lambda') \left[ 1 - e^{-2\gamma\tau} b\psi'/(c\bar\lambda')\right],\\
  \lambda_{\phi\phi} & = (-c\bar\lambda'') \left[ 1 - e^{-2\gamma\tau} b\psi''/(c\bar\lambda'')\right],
\end{align*}
we need to estimate the terms $\psi/\bar\lambda$, $\psi'/\bar\lambda'$, and $\psi''/\bar\lambda''$. We  do this  by considering the asymptotics near $\phi=0$ and near $\phi=\sqrt{2(n-1)}$, respectively.

We first consider the asymptotics as $\phi\nearrow \sqrt{2(n-1)}$. Letting  $c_k$ for $k=0,1,2,\ldots$, denote constants that depend on $n$ and $\gamma$, we determine from \eqref{eq:barlambda} and \eqref{eq:ext-supersub} that as 
$\phi\nearrow \sqrt{2(n-1)}$, we have the following asymptotics
\begin{align*}
\psi/\bar \lambda & = \frac{(2n-2-\phi^2)^{2\gamma-1}}{2\gamma-1} \left[ c_0 + o( 1 ) \right],\\
\psi'/\bar \lambda' & = \frac{(2n-2-\phi^2)^{2\gamma-1}}{(2\gamma-1)^2} \left[ c_1 + o( 1 ) \right],\\
 \psi''/\bar \lambda'' & = \frac{(2n-2-\phi^2)^{2\gamma-1}}{(2\gamma-1)^2\left(1 - n + (-1 + \gamma) \phi^2\right)} \left[c_2 + o(1)\right].
\end{align*}
We note that
\begin{align*}
 \lim\limits_{\phi\nearrow \sqrt{2(n-1)}} \frac{1}{ \left(1 - n + (-1 + \gamma) \phi^2\right)} & = \frac{1}{3-2\gamma}
\end{align*}
is bounded if $\gamma\neq 3/2$; if $\gamma=3/2$, then 
\begin{align*}
 \psi''/\bar \lambda'' =  (2n -2 - \phi^2) \left[ c_3 + o(1) \right], \quad \phi\nearrow \sqrt{2(n-1)}.
\end{align*}
These asymptotics imply that for $ \delta \leq \phi < \sqrt{2(n-1)}$, there exists a constant $M_1$ (independent of $\tau$) such that 
\begin{align}\label{eq:bound-M1}
 \left\vert \psi/\bar \lambda \right\vert, \left\vert \psi'/\bar \lambda' \right\vert, \left\vert \psi''/\bar \lambda'' \right\vert \leq M_1.
\end{align}
Hence we have
\begin{align*}
 II & = - \frac{(\lambda_{\phi\phi}-2\lambda^2_\phi/\lambda)}{e^{-2\gamma\tau} + \lambda^2_\phi/\lambda^4} \\
    & = c^3 \Lambda(\phi) \left\{ 1 +  \pi_3\left( l.o.t. \right)  \right\},
\end{align*}
where is $\pi_3(l.o.t.)$ is a polynomial (without constant term) in $O\left(e^{-2\gamma\tau} b\psi/ (c\bar\lambda) \right)$, $O\left(e^{-2\gamma\tau} b\psi'/(c\bar\lambda')\right)$, and $O\left(e^{-2\gamma\tau} b\psi''/ (c\bar\lambda'')  \right)$. If we now fix $\delta>0$ (e.g., $\delta=1/2$), we see that  for $ \delta \leq \phi < \sqrt{2(n-1)}$, we have
\begin{align}\label{eq:estimate-pi3}
 |\pi_3(l.o.t.)| \leq M_2 e^{-2\gamma\tau}
\end{align}
with $M_2=M_2(b)$ a constant depending (polynomially) only on $b$, which is to be specified. It follows that 
\begin{align*}
e^{2\gamma\tau}\mathcal{F}_\phi[\lambda^+_{ext}] & = II + b \Lambda(\phi) \\
 & \geq \Lambda \left\{ b +  c^3 \left(1 +  M_2 e^{-2\gamma\tau} \right) \right\}.
\end{align*}
We can choose $\tau_2$ sufficiently large so that the inequality $M_2 e^{-2\gamma\tau_2}<1$ holds true for all $\tau\geq\tau_2$. Consequently, for $ \delta \leq \phi < \sqrt{2(n-1)}$, we have 
\begin{align*}
e^{2\gamma\tau}\mathcal{F}_\phi[\lambda^+_{ext}] & > \Lambda \left( b + 2c^3  \right)\\
 & > 0
\end{align*} 
for any $b$ satisfying $b < - 2c^3$.

We next consider the asymptotics as $\phi\searrow 0$. Letting  $d_k$ for  $k=0,1,2,\ldots$, denote constants that depend on $n$ and $\gamma$, we determine from \eqref{eq:barlambda}  and \eqref{eq:ext-supersub}
that as  $\phi\searrow 0$, we have the following asymptotics
\begin{align*}
\psi/\bar \lambda & = c_0 \log(\phi^2) + d_1 + O\left(\phi^2\log(\phi^2)\right),\\
\psi'/\bar \lambda' & = \phi^{-2} \left[ d_2 + O\left(\phi^2\log(\phi^2)\right) \right],\\
 \psi''/\bar \lambda'' & = \phi^{-2} \left[ d_3 + O\left(\phi^2\log(\phi^2)\right) \right].
\end{align*}
These asymptotics imply that for $0<\phi  \leq \delta$, there exists a constant $M_3$ (independent of $\tau$) such that
\begin{align}\label{eq:bound-M2}
 \left\vert \psi/\bar \lambda \right\vert, \left\vert \psi'/\bar \lambda' \right\vert, \left\vert \psi''/\bar \lambda'' \right\vert \leq M_3 \phi^{-2},
\end{align}
and hence we have
\begin{align*}
 II & = - \frac{(\lambda_{\phi\phi}-2\lambda^2_\phi/\lambda)}{e^{-2\gamma\tau} + \lambda^2_\phi/\lambda^4} \\
    & = c^3 \Lambda(\phi) \left\{ 1 +  \pi_4\left( l.o.t. \right)  \right\}
\end{align*}
where is $\pi_4(l.o.t.)$ is a polynomial (without constant term) in $O\left(e^{-2\gamma\tau} b\psi/ (c\bar\lambda) \right)$, $O\left(e^{-2\gamma\tau} b\psi'/(c\bar\lambda') \right)$, and $O\left(e^{-2\gamma\tau} b\psi''/ (c\bar\lambda'')  \right)$, and where 
\begin{align}\label{eq:estimate-pi4}
 |\pi_4(l.o.t.)| \leq M_4 \phi^{-2} e^{-2\gamma\tau}
\end{align}
with $M_4=M_4(b)$ a constant depending (polynomially) only on $b$, which is to be specified. It then follows that
\begin{align}
\label{eq:R2}
M_4\phi^{-2}e^{-2\gamma\tau} & \leq R_2^{-2} M_4;
\end{align}
hence for $R_2e^{-\gamma\tau}\leq \phi<\delta$ (note that by picking $\tau_2$ to be larger if necessary, we always have $R_2e^{-\gamma\tau}<\phi<\delta$ for all $\tau\geq\tau_2$) and for $\tau\geq\tau_2$, we obtain
\begin{align*}
e^{2\gamma\tau}\mathcal{F}_\phi[\lambda^+_{ext}] & = II + b \Lambda(\phi) \\
& > \Lambda(\phi) \left\{ b +  c^3 \left( 1 + R_2^{-2} M_4 \right) \right\} \\
& > 0
\end{align*}
for any $b$ satisfying $b < -c^3 \left( 1 + R_2^{-2} M_4 \right)$.

Therefore, there exists $b^{+} \leq \min \{-2(c^{+})^3, -(c^{+})^3( 1 + R_2^{-2} M_4)\} < 0$ such that $\lambda^{+}_{ext}$ is a supersolution of $\mathcal{F}_\phi[\tilde\lambda]=0$ on the interval $R_2 e^{-\gamma\tau} \leq \phi \leq \sqrt{2(n-1)}$ for all $\tau\geq\tau_2$.

By a similar argument, there exists $b^{-} \geq \max \{-(c^{+})^3/2, -(c^{+})^3( 1 - R_2^{-2} M_4)\}$ such that  $\lambda^{-}_{ext}$ is a subsolution of $\mathcal{F}_\phi[\tilde\lambda]=0$ on the interval $R_2 e^{-\gamma\tau} < \phi < \sqrt{2(n-1)}$ for all $\tau\geq\tau_2$.

The lemma is now proven.

\end{proof}

\begin{remark}
It follows from the proof of Lemma \ref{exterior-supersub} that we can pick $b^- > 0$.
\end{remark}

%%%%%%%%%%%%%%%%%%%%%%%%%%%%%%%%%

%%%%%%%%%%%%%%%%%%%%%%%%%%%%%%%%%

\section{Upper and lower barriers}\label{barriers} 

According to Lemmata \ref{interior-supersub} and \ref{exterior-supersub}, if we choose $R_2<R_1$, then there will be a region where both $\lambda^{\pm}_{int}$ and $\lambda^{\pm}_{ext}$ are defined. To show that the regional supersolutions $\lambda^+_{ext}$ and $\lambda^+_{int}$ together with the regional subsolutions $\lambda^-_{ext}$ and $\lambda^-_{int}$ collectively provide upper and lower barriers for our mean curvature flow problem, we need to show two things:
\begin{itemize}
\item[(i)] in each region, the inequalities $\lambda ^-_{int}\le \lambda ^+_{int}$ and $\lambda ^-_{ext}\le \lambda ^+_{ext}$ hold;
\item[(ii)] that $\lambda ^+_{int}$ and $\lambda ^+_{ext}$ patch together, as do $\lambda ^-_{int}$ and $\lambda ^-_{ext}$.
\end{itemize}

We first prove the regional inequalities (i), via the following two lemmata: 

\begin{lemma}\label{int-order}
For $A^{-}>A^{+}$, there exists $\tau_3\geq\tau_1$ such that
\begin{align*}
 \lambda^{\pm}_{int} := -A^{\pm} + e^{-2\gamma\tau}F(z) + B^{\pm}(\tau) e^{-2\gamma\tau} + D^{\pm} \tau e^{-4\gamma\tau}Q^{\pm}(z)
\end{align*}
satisfy $\lambda^{-}_{int} < \lambda^{+}_{int}$ for $0< |z| <R_1$ and for $\tau\geq\tau_3$.
\end{lemma}

\begin{proof}
Since $Q^{\pm}$ are bounded on $0< |z| \leq R_1$, if $A^{-}>A^{+}$, then it follows from the above expressions that 
\begin{align*}
 \lambda^{+}_{int} - \lambda^{-}_{int} & = A^{-} - A^{+} +  (B^{+}(\tau)-B^{-}(\tau)) e^{-2\gamma\tau} + (E^{+} - E^{-})e^{-2\gamma\tau} + O(\tau e^{-4\gamma\tau}) \\
 & = A^{-} - A^{+} + O(\tau e^{-2\gamma\tau})\\
 & > 0 
\end{align*}
for $\tau\geq \tau_3$ sufficiently large (larger than $\tau_1$ if necessary) and for $0< |z| <R_1$.
\end{proof}

\begin{lemma}\label{ext-order}
 For $c^{-}>c^{+}$, there exists $\tau_4\geq\tau_1$ such that
 \begin{align*}
  \lambda^{\pm}_{ext} & =  -c^{\pm} \bar\lambda(\phi) + b^{\pm} e^{-2\gamma\tau}\psi(\phi)
 \end{align*}
 (see Lemma \ref{exterior-supersub}) satisfy $\lambda^{-}_{ext}\leq \lambda^{+}_{ext}$ for $R_2e^{-\gamma\tau}\leq |\phi| <\sqrt{2(n-1)}$ and for $\tau\geq\tau_4$.
\end{lemma}
\begin{proof}
%Recall the definitions of $\bar\lambda$ and $\psi$ (cf. Lemma \ref{exterior-supersub}). If $c^{+}<c^{-}$, then
Based on the formulas \eqref{eq:ext-supersub} for $\lambda ^+_{ext}$ and $\lambda^-_{ext}$, we have 
\begin{align*}
 \lambda^{+}_{ext} - \lambda^{-}_{ext} & = (c^{-} - c^{+})\bar{\lambda} + (b^+ - b^-) e^{-2\gamma\tau} \psi.
\end{align*}
Then using the expressions \eqref{eq:barlambda} for $\bar\lambda$ and \eqref{eq:soln-psi} for $\psi$, we see that if $|\phi|\in[\delta, \sqrt{2(n-1)})$, then $\psi$ is uniformly bounded, and consequently
\begin{align*}
 \lambda^{+}_{ext} - \lambda^{-}_{ext}  &= (c^{-} - c^{+})\bar{\lambda} + O\left(e^{-2\gamma\tau}\right)\\
  & > 0
\end{align*}
for all sufficiently large $\tau$. If $|\phi|\in [R_2 e^{-\gamma\tau}, \delta]$, then the definition \eqref{eq:soln-psi} of $\psi$ implies that, since $c^->c^+$,
\begin{align*}
 \lambda^{+}_{ext} - \lambda^{-}_{ext}  &= (c^{-} - c^{+})\bar{\lambda} + O\left(\tau e^{-2\gamma\tau}\right)\\
  & > 0
\end{align*}
for all sufficiently large $\tau$. The lemma then follows.
\end{proof}

Recall that Lemma \ref{interior-supersub} holds for any $R_1>0$ and Lemma \ref{exterior-supersub} holds for any $R_2>0$. Below, we  choose $R_2<R_1<<1$ and patch together $\lambda ^+_{int}$ and $\lambda ^+_{ext}$, and $\lambda ^-_{int}$ and $\lambda ^-_{ext}$ in the region defined by $\{ R_2 < z < R_1 \}$. To this end, we need the following lemma:

\begin{lemma}\label{patch}
For a fixed integer $n\geq 2$ and a fixed real number $\gamma > 1/2$, set \\ $\tilde\beta := \left[2(n-1)\right]^{3(\gamma-1/2)}/(2\gamma-1)$ (note that this quantity is postive), let $\lambda^{+}_{int}$ and $\lambda^{-}_{int}$ be as defined in Lemma \ref{int-order}, and let $\lambda^{+}_{ext}$ and $\lambda^{-}_{ext}$ be as defined in Lemma \ref{ext-order}. 
If we choose the constants $A^{\pm}, B^{\pm}, b^{\pm} $ and $c^{\pm}$ satisfying the following relations:
\begin{align}
 A^{\pm}& = c^{\pm}(2n-2)^{\gamma-1/2}>0, \\
 B^{\pm} & = - 2\gamma b^{\pm}\tilde\beta,
\end{align}
and if in particular  we choose $b^{-}>0$, then there exist constants $R_*$ and $R^*$ with \\ $0< R_2 < R_* < R^* < R_1 \ll 1$  such that the following inequalities hold at $z=R_*$ and $z=R^{*}$ for all $\tau\geq \tau_5:=\max\{\tau_3,\tau_4\}$: 
\begin{align}
\lambda^{+}_{int}(R_*,\tau) & < \lambda^{+}_{ext} (R_*,\tau), \quad 
\lambda^{+}_{int}(R^*,\tau)  > \lambda^{+}_{ext} (R^*,\tau), \label{eq:patchupper}\\
\lambda^{-}_{int}(R_*,\tau) & > \lambda^{-}_{ext}(R_*,\tau), \quad
\lambda^{-}_{int}(R^*,\tau) < \lambda^{-}_{ext}(R^*,\tau). \label{eq:patchlower}
\end{align}
Carrying out an even extension (in $z$) for these quantities, we find that corresponding inequalities hold at $z=-R_*$ and $z=-R^*$  for all $\tau\geq\tau_5$.

\end{lemma}

\begin{proof}
We detail the proof for $z, \phi \in (0, 2\sqrt{(n-1)})$; the proof for $z, \phi \in (-2\sqrt{(n-1)}, 0)$ is identical.

We first prove the inequalities in \eqref{eq:patchupper}.

In the interior region, using the asymptotic expansion of $F(z)$ in \eqref{eq:asymp-F}, we have that as $z\searrow 0$,
\begin{align*}
 \lambda^{+}_{int} & = - A^{+} + B^{+}\tau e^{-2\gamma\tau} + e^{-2\gamma\tau}\left[\frac{(\gamma-1/2)A^{+}}{2(n-1)}z^2 + E^{+} + o(z^2)\right] + D^{+} \tau e^{-4\gamma\tau}Q^{+}(z).
\end{align*}
In the exterior region, using the asymptotic expansion that readily follows from the explicit expression for $\psi(\phi)$ in \eqref{eq:soln-psi}, we have that as $\phi\searrow 0$, 
\begin{align*}
 \lambda^{+}_{{ext}} & = -c^{+}[2n-2-z^2e^{-4\gamma\tau}]^{\gamma-1/2} \\
  & \quad + b^{+}e^{-2\gamma\tau}\left[\tilde\beta\log|ze^{-2\gamma\tau}| + d + O\left((ze^{-2\gamma\tau})^2\right)\log|ze^{-2\gamma\tau}| + O\left((ze^{-2\gamma\tau})^2\right)\right]\\
  & = -c^{+}(2n-2)^{\gamma-1/2} + (-2\gamma)b^+ \tau e^{-2\gamma\tau}\left[\tilde\beta + O\left(z^2e^{-4\gamma\tau}\right)\right]\\
  & \quad + b^+ e^{-2\gamma\tau} \left[\tilde\beta\log|z| + d + O\left((ze^{-2\gamma\tau})^2\right)\log|z| + O\left((ze^{-2\gamma\tau})^2\right)\right],
\end{align*}
where $\tilde{\beta}$ (defined above as $\left[2(n-1)\right]^{3(\gamma-1/2)}/(2\gamma-1)$) and $d$ are constants which depend only on $n$ and $\gamma$. We thus have 
\begin{align*}
 \lambda^{+}_{ext} & = -c^{+}(2n-2)^{\gamma-1/2} - (2\gamma b^+\tilde\beta)\tau e^{-2\gamma\tau} + b^+ e^{-2\gamma\tau} \left[\tilde\beta\log|z| + O(1)\right] + O(e^{-4\gamma\tau}). 
\end{align*}
It then follows that
\begin{align*}
 \lambda^{+}_{int} - \lambda^{+}_{ext} & = \left( - A^{+} + c^{+}(2n-2)^{\gamma-1/2} \right) + (B^+ + 2\gamma b^+\tilde\beta)\tau e^{-2\gamma\tau}\\
  & \quad + e^{-2\gamma\tau}\left[\frac{(\gamma-1/2)A^+}{2(n-1)}z^2 - b^+\tilde\beta \log|z| + E^{+} + d + o(z^2) \right]\\
  & \quad + O(\tau e^{-4\gamma\tau}).
\end{align*}
We now recall that our choices of $A^+$ and $c^+$, and of $b^{+}$ and $B^{+}$, are required by hypothesis to  satisfy the following relations:
\begin{align}
 A^{+}& = c^{+}(2n-2)^{\gamma-1/2}, \label{eq:Acplus}\\
 B^+ & = - 2\gamma b^+\tilde\beta.
\end{align} 
We then have 
\begin{align*}
e^{2\gamma\tau}(\lambda^{+}_{int} - \lambda^{+}_{ext}) & = \frac{(\gamma-1/2)A^+}{2(n-1)}z^2 - b^+\tilde\beta \log|z| + E^{+} + d + o(z^2) +  O(\tau e^{-2\gamma\tau}); 
\end{align*}
here we note that for sufficiently large $\tau$, $\left\vert O(\tau e^{-2\gamma\tau})\right\vert << 1$. 

We now make the following observations regarding $\lambda^{+}_{int}$ and  $\lambda^{+}_{ext}$
for any fixed value of $\tau\geq\tau_5$:
\begin{enumerate}
 \item the function $e^{2\gamma\tau}(\lambda^{+}_{int} - \lambda^{+}_{ext})$ is smooth in $z$ and increasing in $z$ if $z$ is small;
 \item since $b^{+}<0$ (cf. the proof of Lemma \ref{exterior-supersub}), it follows that
$\limt{z}{0^+} e^{2\gamma\tau}(\lambda^{+}_{int} - \lambda^{+}_{ext}) = -\infty$ and $\limt{z}{\infty} e^{2\gamma\tau}(\lambda^{+}_{int} - \lambda^{+}_{ext}) = +\infty$; consequently, there exists $z_0$ at which the graph of $e^{2\gamma\tau}(\lambda^{+}_{int} - \lambda^{+}_{ext})$ crosses the $z$-axis for the first time;
 \item since $E^{+}$ is an arbitrary constant (cf. Lemma \ref{interior-supersub}), and since increasing the value of $E^{+}$ slides the graph of $e^{2\gamma\tau}(\lambda^{+}_{int} - \lambda^{+}_{ext})$ upward and shifts $z_0$ towards $0$, we can choose $E^{+}$ so that $z_0\in(R_2, R_1)$.
\end{enumerate}
Thus, there exist $R_3 > 0$ and $K>1$ such that $R_2<R_3<KR_3<R_1<<1$ and such that the following inequalities hold for all $\tau\geq\tau_5$:
\begin{align*}
 e^{2\gamma\tau}(\lambda^{+}_{int} - \lambda^{+}_{ext})(R_3) & < 0, \\
 e^{2\gamma\tau}(\lambda^{+}_{int} - \lambda^{+}_{ext}) (KR_3)& > 0.
\end{align*}

We have thus established the inequalities \eqref{eq:patchupper}. We now proceed to establish the inequalities \eqref{eq:patchlower}:

Using the asymptotic expansion of $F(z)$ as $z\searrow 0$ and that of $\psi(\phi)$ as $\phi\searrow 0$, we obtain
\begin{align*}
 \lambda^{-}_{int} - \lambda^{-}_{ext} & = \left( - A^{-} + c^{-}(2n-2)^{\gamma-1/2} \right) + (B^- + 2\gamma b^-\tilde\beta)\tau e^{-2\gamma\tau}\\
  & \quad + e^{-2\gamma\tau}\left[\frac{(\gamma-1/2)A^-}{2(n-1)}z^2 - b^-\tilde\beta \log|z| + E^{-} + d + o(z^2) \right]\\
  & \quad + O(\tau e^{-4\gamma\tau}).
\end{align*}
Since our choices of  $A^-$ and $c^-$, and $b^{-}$ and $B^{-}$ are required by hypothesis to satisfy
\begin{align}
 A^{-}& = c^{-}(2n-2)^{\gamma-1/2}, \label{eq:Acminus}\\
 B^{-} & = - 2\gamma b^{-}\tilde\beta,
\end{align} 
we determine that
\begin{align*}
e^{2\gamma\tau}(\lambda^{-}_{int} - \lambda^{-}_{ext}) & = \frac{(\gamma-1/2)A^-}{2(n-1)}z^2 - b^-\tilde\beta \log|z| + E^{-} + d + o(z^2) +  O(\tau e^{-2\gamma\tau}),
\end{align*}
where for sufficiently large $\tau$, $\left\vert O(\tau e^{-2\gamma\tau})\right\vert << 1$. We are led to make the following observations for any fixed value of  $\tau\geq\tau_5$:
\begin{enumerate}
 \item the function $e^{2\gamma\tau}(\lambda^{-}_{int} - \lambda^{-}_{ext})$ is smooth in $z$ and increasing in $z$ if $z$ is small;
 \item since $b^{-}>0$ (cf. the proof of Lemma \ref{exterior-supersub}), it follows that  $\limt{z}{0^+} e^{2\gamma\tau}(\lambda^{-}_{int} - \lambda^{-}_{ext}) = +\infty$ and $\limt{z}{\infty} e^{2\gamma\tau}(\lambda^{-}_{int} - \lambda^{-}_{ext}) = -\infty$; consequently there exists $z_1$ at which  the graph of $e^{2\gamma\tau}(\lambda^{-}_{int} - \lambda^{-}_{ext})$ crosses the $z$-axis for the first time;
 \item since  $E^{-}$ is an arbitrary constant (cf. Lemma \ref{interior-supersub}), and since decreasing the value of $E^{-}$  slides the graph of $e^{2\gamma\tau}(\lambda^{-}_{int} - \lambda^{-}_{ext})$ downward and shifts $z_1$ towards $0$, we can choose $E^{-}$ such that $z_1\in(R_2, R_1)$.
\end{enumerate}
Thus, there exist $R_4 > 0$ and $L>1$ such that $R_2<R_4<LR_4<R_1<<1$ and such that the following inequalities hold for all $\tau\geq\tau_5$:
\begin{align*}
 e^{2\gamma\tau}(\lambda^{-}_{int} - \lambda^{-}_{ext}) (R_4) & > 0,\\
 e^{2\gamma\tau}(\lambda^{-}_{int} - \lambda^{-}_{ext}) (L R_4) & < 0.
\end{align*}

If we now set $R_* := \min\{ R_3, R_4 \}$ and $R^* := \max\{ K R_3, L R_4 \}$, the proof of this lemma is complete.
\end{proof}

\begin{remark}
The choices of $A^{\pm}$ and of $c^{\pm}$ in Lemma \ref{patch} guarantee that Lemmata \ref{int-order} and \ref{ext-order} both hold.
\end{remark}

We now patch the regional sub and super solutions, thereby producing global sub and super solutions, which are consequently upper and lower barriers: For $|\phi|\in[0,\sqrt{2(n-1)})$ and for $\tau\geq\tau_5$, we define $\lambda^{+}:=\lambda^{+}(\phi,\tau)$ by
\begin{align}\label{eq:upperbarrier}
 \lambda^{+} : = \left\{
  \begin{array}{cc}
   \lambda^{+}_{int}, & |\phi|\leq R_* e^{-\gamma\tau} \vspace{6pt} \\ 
   \inf\left\{\lambda^{+}_{int}, \lambda^{+}_{ext} \right\}, & R_* e^{-\gamma\tau}\leq |\phi| \leq R^* e^{-\gamma\tau} \vspace{6pt}\\
   \lambda^{+}_{ext} , & R^* e^{-\gamma\tau} \leq |\phi| < \sqrt{2(n-1)}
  \end{array}
 \right.,
\end{align}
and we define
$\lambda^{-}:=\lambda^{-}(\phi,\tau)$ by
\begin{align}\label{eq:lowerbarrier}
 \lambda^{-} : = \left\{
  \begin{array}{cc}
   \lambda^{-}_{int}, & |\phi|\leq R_* e^{-\gamma\tau} \vspace{6pt}\\
   \sup\left\{\lambda^{-}_{int}, \lambda^{-}_{ext} \right\}, & R_* e^{-\gamma\tau}\leq |\phi| \leq R^* e^{-\gamma\tau} \vspace{6pt}\\
   \lambda^{-}_{ext} , & R^* e^{-\gamma\tau} \leq |\phi| < \sqrt{2(n-1)}
  \end{array}
 \right..
\end{align}

The properties of the barriers $\lambda^{\pm}$ are summarized in the following Proposition.
\begin{prop}\label{prop-patch}
For a fixed  integer $n\geq 2$ and a fixed real number $\gamma > 1/2$, let $\lambda^{+}$ and $\lambda^{-}$ be defined as in \eqref{eq:upperbarrier} and \eqref{eq:lowerbarrier}, respectively. There exists a sufficiently large (but finite) $\tau_0$ such that the following hold true for $-\sqrt{2(n-1)}<\phi<\sqrt{2(n-1)}$ and for $\tau\geq\tau_0$:
\begin{enumerate}
\item[(B1)] $\lambda^{+}$ and $\lambda^{-}$ are super- ($+$) and sub- ($-$) solutions for equation \eqref{eq:lambda(phi,tau)}, resp., for \\ $-\sqrt{2(n-1)} <\phi<\sqrt{2(n-1)}$ and $\tau\geq\tau_0$. \\

\item[(B2)] $\lambda^{-} < \lambda^{+}$ for $(\phi,\tau)\in (-\sqrt{2(n-1)},\sqrt{2(n-1)})\times(\tau_0,\infty)$. \\
\item[(B3)] Near $\phi = 0$ one has $\lambda^{\pm} = \lambda^{\pm}_{int}$; and near $\phi=\sqrt{2(n-1)}$ one has $\lambda^{\pm} = \lambda^{\pm}_{ext}$. \\
\item[(B4)] For any $\tau\in[\tau_0,\infty)$, $\lim\limits_{|\phi|\nearrow \sqrt{2(n-1)}} \lambda^{\pm} = 0$.
\end{enumerate}
\end{prop}

\begin{proof}
Let $\tau_0\geq \tau_5$. Condition (B1) follows from Lemmata \ref{exterior-supersub} and \ref{interior-supersub}. Condition (B2) follows from the definition of $\lambda^{\pm}$ and Lemmata \ref{int-order} and \ref{ext-order}. Condition (B3) follows from the definition of $\lambda^{\pm}$. Condition (B4) follows from the limit  $\lim\limits_{|\phi|\nearrow \sqrt{2(n-1)}} \lambda^{\pm}_{ext} = 0$. 
\end{proof}

We now prove a comparison principle for any pair of smooth functions such that one of them is a subsolution 
of equation $\mathcal{F}_\phi[\lambda] = 0$ (cf. \eqref{eq:lambda(phi,tau)}) and the other is a smooth supersolution of the same equation (they need not be related to the subsolution $\lambda^{-}$ or supersolution $\lambda^{+}$ constructed above).

\begin{prop}\label{comparison}{(Comparison principle for $\mathcal{F}_{\phi}[\lambda]=0$)}
For a fixed  integer $n\geq 2$, a fixed real number $\gamma > 1/2$, and some $\bar\tau\in(\tau_0, \infty)$, suppose that $\zeta^{+}$, $\zeta^{-}$ are any smooth non-positive super- and sub- solutions (not necessarily those constructed in Proposition \ref{prop-patch}) of the equation $\mathcal{F}_{\phi}[\lambda]=0$, respectively. Assume that
\begin{itemize}
 \item[(C1)] $\zeta^{-}(\phi,\tau_0) < \zeta^{+}(\phi,\tau_0)$ for $\phi\in(-\sqrt{2(n-1)},\sqrt{2(n-1)})$,\\
 
 \item[(C2)] $\zeta^{-}(-\sqrt{2(n-1)},\tau) \leq \zeta^{+}(-\sqrt{2(n-1)},\tau)$ for $\tau\in[\tau_0,\bar\tau]$,\\
 
 \item[(C3)] $\zeta^{-}(\sqrt{2(n-1)},\tau) \leq \zeta^{+}(\sqrt{2(n-1)},\tau)$ for $\tau\in[\tau_0,\bar\tau]$, 
\end{itemize}

Then $\zeta^{-}(\phi,\tau_0) \leq \zeta^{+}(\phi,\tau_0)$ for $(\phi,\tau)\in [-\sqrt{2(n-1)},\sqrt{2(n-1)}] \times [\tau_0,\bar\tau]$.
\end{prop}

\begin{proof}
Let $\epsilon>0$ be arbitrary and define $v:= e^{-\mu \tau}(\zeta^+ - \zeta^{-}) + \epsilon$, where $\mu$ is to be determined.

We claim that $v\geq 0$ on $(\phi,\tau)\in [-\sqrt{2(n-1)},\sqrt{2(n-1)}] \times [\tau_0,\bar\tau]$. To prove this, suppose the contrary.  Then it follows from assumptions (C1)--(C3) that there must be a first time $\tau_*\in(\tau_0,\bar\tau)$ and an interior point $\phi_*\in(-\sqrt{2(n-1)},\sqrt{2(n-1)})$ such that
\begin{align*}
 v(\phi_*, \tau_*) = 0.
\end{align*}
Moreover, at $(\phi_*,\tau_*)$, we have
\begin{align*}
 \p_\tau\vert_{\phi} v  \leq 0, &  \quad\quad\quad \zeta^{+}_{\phi\phi}  \geq \zeta^{-}_{\phi\phi}, \\
 \zeta^{+}_{\phi}  = \zeta^{-}_{\phi}, & \quad\quad\quad \zeta^{+} - \zeta^{-}  =  -\epsilon e^{\mu \tau_*}.
\end{align*}
Consequently at $(\phi_*,\tau_*)$, we have
\begin{align*}
 0 & \geq e^{\mu \tau_*} \p_\tau\vert_{\phi} v \\
  & = \p_\tau\vert_{\phi} (\zeta^{+} - \zeta^{-})  - \mu \tau_* (\zeta^{+}-\zeta^{-})\\
  & \geq \frac{(\zeta^+_{\phi\phi}-2(\zeta^{+}_\phi)^2/\zeta^+)}{1+e^{2\gamma\tau_*}(\zeta^{+}_\phi)^2/(\zeta^{+})^4} - \frac{(\zeta^-_{\phi\phi}-2(\zeta^{-}_\phi)^2/\zeta^-)}{1+e^{2\gamma\tau_*}(\zeta^{-}_\phi)^2/(\zeta^{-})^4}\\
  & \quad + \left(\frac{n-1}{\phi_*}-\frac{\phi_*}{2}\right)\left(\zeta^+_\phi - \zeta^-_\phi\right) + (\gamma-1/2)(\zeta^{+} - \zeta^-) \\
  &\quad - \mu \tau_* (\zeta^{+}-\zeta^{-})\\
  & \geq (\zeta^+ - \zeta^-) \left\{ \left. \frac{2(\zeta^+_\phi)^2}{\zeta^+\zeta^-(1+e^{2\gamma\tau}(\zeta^{+}_\phi)^2/(\zeta^{+})^4)}\right\vert_{(\phi_*,\tau_*)} + (\gamma-1/2) - \mu \tau_*  \right\}\\
  & = -\epsilon e^{\mu \tau_*} \left\{ \text{bounded term} - \mu \tau_* \right\}.
\end{align*}
We recall that $\epsilon>0$ is fixed. If we choose $\mu$ sufficiently large, then  at $(\phi_*,\tau_*)$ we have 
\begin{align*}
 0 & \geq \p_\tau\vert_{\phi} v > 0,
\end{align*}
which is a contradiction. Hence, the claim is true. Since $\epsilon>0$ is arbitrary, the proposition follows.
\end{proof}

\begin{remark}
Despite the fact that the piecewise smooth upper barrier $\lambda^{+}$ and the piecewise smooth lower barrier $\lambda^{-}$ defined by \eqref{eq:upperbarrier} and \eqref{eq:lowerbarrier}, respectively, are not smooth, 
the comparison principle (Proposition \ref{comparison}) applies to them. This is because, by construction, the non-smooth corners of  $\lambda^{+}$ are convex and the non-smooth corners of $\lambda^{-}$ are concave. Hence, the points of first contact of either $\lambda^{+}$ or $\lambda^{-}$ with a given smooth MCF solution $\lambda$ are necessarily away from the corners; they occur at smooth points of $\lambda^{+}$ or $\lambda^{-}$.
\end{remark}

We end this section by discussing the relation between the barriers $\lambda^+$ and $\lambda^-$ and a formal solution $\lambda_{form}$. Suppose that $c\in(c^+, c^-)$, and let $A=c(2n-2)^{\gamma-1/2}$. It then follows from equations \eqref{eq:Acplus} and \eqref{eq:Acminus} that we have $A\in (A^+, A^-)$. Now consider the following formal solutions defined in the interior and exterior regions, respectively, for all $\tau\geq\tau_5$:
\begin{align*}
 \lambda_{form, int}(z,\tau) & = -A + e^{-2\gamma\tau}F(z),   \quad\quad\quad |z|\in[0, R^*];\\
  \lambda_{form, ext}(\phi,\tau) & = -c (2n-2-\phi^2)^{\gamma-1/2}, \quad |\phi|\in[R_*e^{-\gamma\tau},\sqrt{2(n-1)}).
\end{align*}
We see that (cf. the proof of Lemmata \ref{int-order} and \ref{ext-order}) for all $\tau\geq\tau_5$,
\begin{align*}
 \lambda^{-}_{int} & < \lambda_{form, int} < \lambda^{+}_{int}, \quad |z|\in[0, R^*];\\
 \lambda^{-}_{ext} & < \lambda_{form, ext} < \lambda^{+}_{ext}, \quad |\phi|\in[R_*e^{-\gamma\tau},\sqrt{2(n-1)}).
\end{align*}

%%%%%%%%%%%%%%%%%%%%%%%%%%%%%%
%%%%%%%%%%%%%%%%%%%%%%%%%%%%%%

\section{Proof of the main theorem}\label{existence}

We first show that the highest curvature for a convex rotationally symmetric solution of MCF is always achieved at the tip.\footnote{The proof of the lemma is contained in \cite[Lemma 3.1]{ADS15}. We include it here for completeness.}
\begin{lemma}\label{tip}
For a convex rotationally symmetric solution $\Gamma_t$, $t\in(0,T)$, of mean curvature flow, the maximum curvature
$\sup\limits_{p\in M_t}|h(p,t)|$ (and hence $\sup\limits_{p\in M_t}H(p, t)$) occurs at the tip of the surface $\Gamma_t$ for each $t\in(0,T)$
\end{lemma}
\begin{proof}
This fact is a consequence of the convexity of the graph, which is preserved under MCF. The principal curvatures of the rotationally symmetric graph are
\begin{align*}
\kappa_1 & = \cdots = \kappa_{n-1} = \frac{1}{u (1+u_x^2)^{1/2}},\quad \kappa_n  = -\frac{u_{xx}}{(1+u_x^2)^{3/2}}.
\end{align*}
From these formulas, we see that if we define the quantity  $R  := \frac{\kappa_n}{\kappa_1}$, we have 
\begin{align*}
 R & = - \frac{u u_{xx}}{(1+u_x^2)} = (1-n) - u u_t,
\end{align*}
where the last equality follows from the MCF equation \eqref{eq:u(x,t)}. Note that $R\geq 0$ because $u_{xx}(\cdot,t)\leq 0$ as a consequence of the presumed convexity of the graph. One computes that
\begin{align*}
 R_t & = \frac{R_{xx}}{1+u_x^2}  - \frac{2u_x}{u(1+u_x^2)}(1-R)R_x + \frac{2u_x^2}{u^2(1+u_x^2)}\left[(1-R^2)-(n-2)(1-R)\right].
\end{align*}

For $t\in(0,T)$, $R=1$ at the tip since the tip is an umbilic point, and $\lim\limits_{x\nearrow\infty}R=0$ since the surface is asympotic to a shrinking cylinder. At a maximum of $R$ (which occurs at an interior point), $R_x=0$, $R_{xx}\leq 0$; hence 
\begin{align*}
 \p_t R_{\max} \leq \frac{2u_x^2}{u^2(1+u_x^2)}\left[(1-R_{\max}^2)-(n-2)(1-R_{\max})\right].
\end{align*}
It then follows from the maximum principle that $R\leq 1$. Note that $R\leq 1$ is equivalent to the inequality $1+u_x^2+u u_{xx} \geq 0$.

For $i=1,\ldots, n-1$, we have $\kappa_i^{-1} = u(1+u_x^2)^{1/2}$, which implies that
\begin{align*}
 \left(\kappa_i^{-1}\right)_x & = u_x(1+u_x^2)^{-1/2}(1+u_x^2+u u_{xx}) \geq 0
\end{align*}
since $u_x\geq 0$ for $x\geq x_0(t)$ (the position of the moving tip). So $\kappa_i$ has its maximum at the tip. Since $R=1$ at the tip and $R\leq 1$ everywhere, all the principal curvatures are maximal at the tip. Therefore, $\sup\limits_{p\in M_t}|h(p,t)| = \sup \sqrt{\lambda_1^2 + \cdots + \lambda_n^2}$ and $\sup\limits_{p\in M_t} H(p,t) = \sup (\lambda_1 + \cdots + \lambda_n)$ occur at the tip.
\end{proof}

We now prove the main theorem of the paper.

\begin{proof}[Proof of Theorem \ref{thmmain}] Let $n\geq 2$ and $\gamma>1/2$. Let $\tau_0\geq \tau_5$, where $\tau_5$ is given in Lemma \ref{patch}.

We first construct initial data for MCF flow by patching formal solutions in the interior and exterior regions at $\tau=\tau_0$. Pick $c\in(c^{+}, c^{-})$, and let $A = c(2n-2)^{\gamma-1/2}$. It follows that $A\in(A^{+}, A^{-})$. Recalling  that $z=\phi e^{-\gamma\tau}$, we define
\begin{align*}
\hat\lambda_0(\phi) := \left\{
\begin{array}{lr}
\begin{array}{l} -A + e^{-2\gamma\tau_0}F(z)- e^{-2\gamma\tau_0}F(R_*) \\
 \quad + \left[ A -c\left(2n-2-(R_*e^{-\gamma\tau_0})^2\right)^{\gamma-1/2} \right] 
 \end{array}, & 0\leq|z|\leq R_*,\\ \\
-c(2n-2-\phi^2)^{\gamma-1/2}, & R_*e^{-\gamma\tau_0}\leq|\phi|<\sqrt{2(n-1)},
\end{array}
\right.
\end{align*}
where we recall that $R_*$ is defined in Lemma \ref{patch}. In particular, by taking $\tau_0$ sufficiently large, we can ensure
\begin{align*}
 \left\vert - e^{-2\gamma\tau_0}F(R_*) + A-c\left(2n-2-(R_*e^{-\gamma\tau_0})^2\right)^{\gamma-1/2}\right\vert < \min\left\{\frac{A^{-} - A}{100},\frac{A - A^{+}}{100}\right\},
\end{align*}
so then at $\tau=\tau_0$ and for all $|\phi|\in[0,\sqrt{2(n-1)})$, we have
\begin{align*}
 \lambda^{-}(\phi,\tau_0) < \hat\lambda_0(\phi) < \lambda^{+}(\phi,\tau_0).
\end{align*}
It follows from its construction that $\hat\lambda_0$ is continuous and piecewise smooth, and that \\ $\lim\limits_{|\phi|\nearrow \sqrt{2(n-1)}} \hat\lambda_0 = 0$. Moreover, $\hat\lambda$ is convex since $F(z)$ and $-c (2n-2-\phi^2)^{\gamma-1/2}$ are convex. Consequently  we can smooth $\hat\lambda_0$ to obtain an open set of smooth convex functions such that each such function $\lambda_0$ has the properties that $\lambda^{-} < \lambda_0 < \lambda^{+}$ for $|\phi|\in[0,\sqrt{2(n-1)})$ at $\tau=\tau_0$, and that $\lim\limits_{|\phi|\nearrow \sqrt{2(n-1)}} \lambda_0 = 0$. Then this open set of functions $\lambda_0$ corresponds to an open set $\mathscr{G}$ of smooth, complete, rotationally symmetric, convex, asymptotically cylindrical hypersurfaces over a ball.

MCF starting from a hypersurface $\Gamma\in\mathscr{G}$ must have a smooth solution up to its vanishing time $T$. Let $\lambda(\phi,\tau)$ correspond to such a MCF solution. Since $\lambda^{-} < \lambda < \lambda^{+}$ on $(-\sqrt{2(n-1)},\sqrt{2(n-1)})$ at $\tau=\tau_0$, the comparison principle (Proposition \ref{comparison}) implies that the solution is always trapped between the barriers; that is,  $\lambda^{-} < \lambda < \lambda^{+}$ on $(-\sqrt{2(n-1)},\sqrt{2(n-1)})\times (\tau_0,\infty)$. In particular, the asymptotics of $\lambda^{-}$ and $\lambda^{+}$ as \\ $\phi\nearrow\sqrt{2(n-1)}$ imply that for a fixed choice of $\gamma>1/2$, $\lambda(\phi,\tau) \sim (2n-2-\phi^2)^{\gamma-1/2}$ as $|\phi|\nearrow\sqrt{2(n-1)}$ for all $\tau\geq\tau_0$.  This  implies part (3) of Theorem \ref{thmmain}.

\begin{remark}
For the class of MCF solutions we consider, the asymptotic cylindrical condition is given by a precise rate and this rate is preserved under MCF.
\end{remark}

To study the behavior of such a MCF solution near the tip as $\tau\nearrow\infty$, we work with $y(z,\tau)$ instead of $\lambda(z,\tau)$. Recall that $y(z,\tau)$ evolves by equation \eqref{eq:y(z,tau)}. Let $\tilde A = -1/A$. Define $\tilde p(z,\tau)$ by the relation
\begin{align}
 y(\phi,\tau) = \tilde A + e^{-2\gamma\tau} \tilde p(z,\tau).
\end{align}
Then $\tilde p(z,\tau)$ satisfies the PDE $\mathcal{B}[\tilde p] = 0$ where
\begin{align}
 \mathcal{B}[\tilde p] & = (\gamma-1/2)\tilde A - \left(\frac{\tilde{p}_{zz}}{1+\tilde{p}^2_z} + \frac{n-1}{z}\tilde p_{z}\right) + e^{-2\gamma\tau}\left\{ \left.\p_\tau\right\vert_z \tilde p + (\gamma+1/2)(z\tilde p - \tilde p) \right\}.
\end{align}
Part (2) of Theorem \ref{thmmain} is a consequence of the following lemma:
\begin{lemma}[Type-II blow-up]\label{lem:type2}
Recall the function $\tilde P$ defined in \eqref{eq:tildeP} which forms part of a formal solution to MCF. We have the following asymptotic behavior of $\tilde p$:
\begin{align}
\lim\limits_{\tau\nearrow\infty} \left( \tilde p(z,\tau) - \tilde p(0,\tau) \right) = \frac{1}{(\gamma-1/2)\tilde A} \tilde P\left( (\gamma-1/2)\tilde A z \right)
\end{align}
uniformly on compact $z$ intervals.
\end{lemma}
\begin{proof}[Proof of Lemma \ref{lem:type2}]

We show that $\tilde p(z,\tau)$ converges uniformly to $\tilde{P}\left((\gamma-1/2)\tilde A z \right)$ as $\tau\to\infty$ for bounded $z\geq 0$. To do this, it is useful to define a new ``time'' variable $s=e^{2\gamma\tau}/(2\gamma)$. In terms $s$,  $\tilde p$ satisfies the PDE
\begin{align}\label{eq:p_s}
 \left. \p_s \right\vert_z \tilde p & = \frac{\tilde{p}_{zz}}{1+\tilde{p}^2_z} + \frac{n-1}{z}\tilde p_{z} - (\gamma-1/2)\tilde A + \frac{2\gamma+1}{4\gamma} \frac{1}{s} (\tilde p - z\tilde p).
\end{align}
The quantity $q(z,s):=\tilde{p}_z(z,\tau)$ satisfies $\mathcal{P}[q]=0$, where
\begin{align}\label{eq:P[q]}
 \mathcal{P}[q] & = \frac{\p q}{\p s} + \frac{2\gamma+1}{4\gamma} \frac{1}{s} z q_z - \frac{\p}{\p z}\left( \frac{q_z}{1+q^2} + \frac{n-1}{z} q \right).
\end{align} 
We note that equations \eqref{eq:p_s} and \eqref{eq:P[q]} are the same as equations (7.13) and (7.14) in \cite{AV97}. Indeed, we see that the coefficient $\frac{2\gamma+1}{4\gamma}$ here is replaced by $\frac{m-1}{m-2}$ in \cite{AV97}. However, we recall that in \cite{AV97}, $\gamma=\frac{1}{2}-\frac{1}{m}$, so $\frac{2\gamma+1}{4\gamma} = \frac{m-1}{m-2}$. Therefore, the rest of the proof in \cite[pp.51--58]{AV97} applies to our case \emph{mutatis mutandis}.
\end{proof}

Lemma \ref{lem:type2} implies that a smooth convex MCF solution expressed in $y(z,\tau)$ satisfies the following asymptotics: on a compact $z$ interval (in the interior region), as $\tau\nearrow\infty$, 
\begin{align*}
 y (z,\tau) & = \tilde A - e^{-\gamma\tau} \tilde{p}(0,\tau) + e^{-2\gamma\tau} \frac{1}{(\gamma-1/2)\tilde A} \tilde P\left( (\gamma-1/2)\tilde A z \right)\\
 & = y(0,\tau) + e^{-2\gamma\tau} \frac{1}{(\gamma-1/2)\tilde A} \tilde P\left( (\gamma-1/2)\tilde A z \right).
\end{align*}
The highest curvature of this MCF solution occurs at the tip where $z=0$ (cf. Lemma \ref{tip}) and necessarily blows up at the rate predicted by the formal solution $e^{-2\gamma\tau} \frac{\tilde P\left( (\gamma-1/2)\tilde A z \right)}{(\gamma-1/2)\tilde A}$ (cf. Section \ref{formal}), which, after being translated back in the $(x,t)$-coordinates (cf. \eqref{eq:meancurv}), proves part (1) of Theorem \ref{thmmain}.

Therefore, Theorem \ref{thmmain} is proved.
\end{proof}

\subsection*{Acknowledgments}
We are grateful to Sigurd Angenent and Dan Knopf for their interest in and helpful discussions on this project. We also thank the referee for valuable comments on our manuscript.

\bibliography{mcf-type2_bib}
\bibliographystyle{crelle}

%\begin{thebibliography}{9}
%\end{thebibliography}
\end{document}